\documentclass[11pt,a4paper]{article}

\usepackage{amsfonts} 
\usepackage{caption} 
\usepackage{mathtools} 
\usepackage{amsthm}
\usepackage{amsmath} 
\usepackage{graphicx} 
\usepackage[colorlinks=false]{hyperref} 
\usepackage{mathbbol} 
\usepackage[title]{appendix} 
\usepackage{tikz-cd}  
\usepackage{comment}
\usepackage{float} 
\newtheorem{theorem}{Theorem}[section] 
\newtheorem{definition}[theorem]{Definition} 
\newtheorem{remark}[theorem]{Remark}

\newtheorem{proposition}[theorem]{Proposition}
\newtheorem{corollary}[theorem]{Corollary}
\newtheorem{lemma}[theorem]{Lemma}
\newtheorem{notation}[theorem]{Notation} 

\usepackage{cleveref}
\usepackage[small]{titlesec} 

\newcommand{\footremember}[2]{%
	\footnote{#2}
	\newcounter{#1}
	\setcounter{#1}{\value{footnote}}%
}

\makeatletter
\def\blfootnote{\gdef\@thefnmark{}\@footnotetext}
\makeatother
\begin{document}
\title{Scaling limit of the collision measures of multiple random walks}
\author{Dinh-Toan Nguyen}

\author{
	Dinh-Toan Nguyen\footremember{alley}{ LAMA, Univ Gustave Eiffel, Univ Paris Est Creteil, CNRS, F-77454 Marne-la-Vallée, France.} \footremember{trailer}{Département de Mathématiques, Université du Québec à Montréal (UQAM), Montréal, QC H2X 3Y7, Canada.}
}
\date{}
\blfootnote{\textit{MSC2020 subject classifications:} 60F05, 60G57, 82C05.} 
\maketitle
\begin{abstract}\noindent
For an integer $k\ge 2$, let $S^{(1)}, S^{(2)}, \dots,  S^{(k)}$ be $k$ independent simple symmetric random walks on $\mathbb{Z}$. A pair $(n,z)$ is called a collision event if there are at least two distinct random walks, namely, $S^{(i)},S^{(j)}$ satisfying $S^{(i)}_n= S^{(j)}_n=z$. We show that under the same scaling as in Donsker's theorem, the sequence of random measures representing these collision events converges to a non-trivial random measure on $[0,1]\times \mathbb{R}$. Moreover, the limit random measure can be characterized using Wiener chaos. The proof is inspired by methods from statistical mechanics, especially, by a partition function that has been developed for the study of directed polymers in random environments.\\

\end{abstract}
\pagenumbering{roman}
\pagenumbering{arabic}

\section{Introduction}
\label{section: Introduction}
For an integer $k\ge 2$, let $S^{(1)}, S^{(2)}, \dots,  S^{(k)}$ be $k$ independent simple symmetric random walks (SSRWs) on $\mathbb{Z}$, defined on a probability space $(\Omega, \mathcal{A}, \mathbf{P})$ (see \cite[p.3]{Pal2005}). A pair $(n,z) \in \mathbb{N}\times \mathbb{Z}$ is called \textit{a collision event} if there are at least two random walks that collide (occupy the same position at the same time) at the time $n$ and the location $z$ (see Figure \ref{fig:collision1}), and $n$ is then called a \textit{collision time}.\\
First mentioned in Pólya's note \cite{Polya1984}, the collision of random walks has since then been a classic topic in probability theory. Recently, this topic has gained more attention from researchers working on the random walks on graphs \cite{Barlow2012,Hutchcroft2015} and random environments \cite{Halberstam2022,Chen2016,Avena2018}.\\
When consider only two random walks, collision problems are strongly related to Brownian local time \cite{Knight1963,Revesz1981,Szabados2005}. The convergence of collision times can be achieved by coupling a new SSRW, the difference between two given random walks, with a Brownian motion using Skorokhod's embedding \cite[p.52]{Pal2005}\cite[Theorem 8.6.1]{Durrett2010}. However, these methods cannot be easily adapted to give convergence results for the collisions of more than two SSRWs because the couplings rely heavily the choices of stopping times which are proper for each SSRW.\\
To the best of our knowledge, scaling limit results for collisions of $k>2$ random walks are still limited.\\
In this paper, we investigate a relatively uncommon aspect of random walk collisions, concerning their duality with the partition function of a directed polymer model in statistical mechanics \cite{Carmona2002}. By following the ideas developped in \cite{Carmona2002,TomAlberts2014}, we obtain new results on when and where the collisions of these random walks occur after long observation, or more precisely, on the scaling limit of the empirical measures of the collision events.
\begin{figure*}[h!]
	\caption{An example of collision events when $k=3, N=100$. The collision events are represented by blue dots.}
	\centering
	\includegraphics[width=0.5\textwidth]{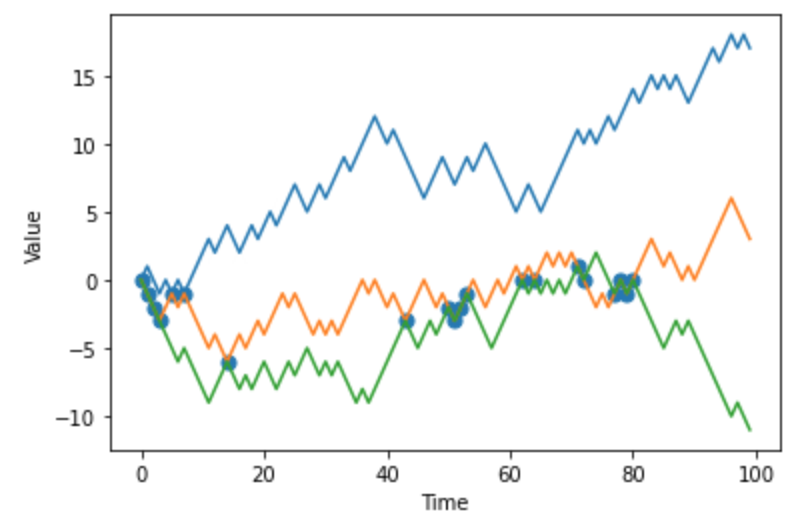}
	\label{fig:collision1}	
\end{figure*}
\\Our study objects are as follows:
\begin{definition} For each $N \in \mathbb{N}$, we define the collision measures of $k$ random walks $S^{(1)},S^{(2)},\dots,S^{(k)}$ until time $N$ to be:
	\label{definition: empirical measures}
	\begin{align*}
		&\Pi_N := \sum_{n=1}^N \sum_{z \in \mathbb{Z}} \sum_{ \substack{1 \le i < j \le k \space:\\ S_n^{(i)}= S_n^{(j)}=z}} \delta_{ \left(\frac{n}{N}, \frac{z}{\sqrt{N}}\right)}, \text{ and}
		\\
		&\Pi'_N :=\sum_{n=1}^N \sum_{z \in \mathbb{Z}}   \delta_{ \left(\frac{n}{N}, \frac{z}{\sqrt{N}}\right)} \mathbb{1}_{\{(n,z)\text{ is an collision event} \}},
	\end{align*}
where $\delta$ is the Dirac measure.
\end{definition}
For each $N \in \mathbb{N}$, the main difference between $\Pi'_N$ and $\Pi_N$  is that $\Pi_N$ takes into account the multiplicity of collision events. For example, if the number of considered random walks is $3$ and it happens that $S^{(1)}_n=S^{(2)}_n=S^{(3)}_n=z$ for some $(n,z)$, then the Dirac measure $\delta_{(\frac{n}{N},\frac{z}{\sqrt{N}} )}$ will appear $3$ times in the summation in $\Pi_N$ while the number of its appearance in $\Pi'_N$ is still one. \\
Concerning the scaling choice, one can observe that this is the same scaling as in Donker's theorem, which also suggests that the distribution of $\hat{\Pi}_N$ is closely related to Brownian local time. Indeed, when $k=2$, the total measure of $\Pi_N$ is equal to 
\begin{equation} \# \{ n \in [\![1,N]\!]: S^{(1)}_n-S^{(2)}_n=0 \} \stackrel{(d) }{=} \# \{ n \in [\![1,N]\!]: S^{(1)}_{2n}=0 \} = L^0_{S^{(1)}}\left( 2N \right),
\label{equation: local time k=2}
\end{equation}
where $\stackrel{(d) }{=}$ is the equality in law, and $L^0_S(t)$ is the local time at position $0$ during the period $[0,t]$ of some walk $S$. So by the convergence of local time of simple random walks, the equation \eqref{equation: local time k=2} implies $\frac{1}{\sqrt{N}} \left\| \Pi_N \right\| =O_{\mathbf{P}}(1)$, where $\|\Pi_N\|$ denotes the total measure of $\Pi_N$.\\
In this work, we not only bound the sequence $( \frac{1}{\sqrt{N}} \Pi_N   ; N \in \mathbb{N})$, but also prove that this sequence of random measures converges to a non-trivial random measure $\mathcal{N}$ on $[0,1]\times \mathbb{R}$. Before giving our main result, let us recall the convergence of random measures.
\begin{definition}(Convergence of random measures)
	\label{definition: Convergence of random measures}
	Suppose $\xi,\xi_1,\xi_2,\dots$ are random finite measures on $([0,1]\times \mathbb{R},\mathcal{B}([0,1]\times \mathbb{R}))$, we say that $\xi_n \xrightarrow[n \rightarrow \infty]{wd} \xi $ if the sequence of real random variables $(\xi_n(f),n \in \mathbb{N})$ converges in distribution to $\xi(f)$ when $n$ goes to infinity for all bounded continuous function $f \in \mathcal{C}_b([0,1]\times \mathbb{R})$. Here,  $\mu(f)$ denotes the integral  $\int f d\mu$ for any (random) measure $\mu$ and bounded measurable function $f$.
\end{definition}
Here are our main results:
\begin{theorem}(Convergence of collision measures and characterization of the limit random measure)\label{theorem: main theorem}
	\begin{itemize}
		\item 
		There is a random finite positive measure $\mathcal{N}$ on the measurable space $([0,1]\times \mathbb{R}, \mathcal{B}([0,1]\times \mathbb{R}) )$ such that:
		\begin{equation*}
			\frac{1}{\sqrt{N}}\Pi_N \xrightarrow[N \rightarrow+\infty]{wd} \mathcal{N} \qquad \text{ and } \qquad 		\frac{1}{\sqrt{N}}\Pi'_N \xrightarrow[N \rightarrow+\infty]{wd} \mathcal{N}.
		\end{equation*}
		\item Furthermore, for all nonnegative bounded continuous function $f \in \mathcal{C}_{b,+}([0,1]\times \mathbb{R})$, the exponential moment of $\mathcal{N}$ with respect to $f$ is equal to $k$-th moment of a positive random variable $\mathcal{Z}_{\sqrt{2f}}$:
		\begin{equation*}
			\mathbf{E}\left[e^{  \mathcal{N}(f)} \right] =\mathbf{E} \left[\left( \mathcal{Z}_{\sqrt{2f}} \right)^k \right].
		\end{equation*}
		where for each $a \in \mathcal{C}_{b,+}([0,1]\times \mathbb{R})$, the random variable $\mathcal{Z}_a$  is identified as the sum of multiple stochastic integrals given by:
		\begin{equation}\label{equation: Z_a}
			\mathcal{Z}_{a} := 1+\sum_{n=1}^{\infty} \int_{\Delta_n} \int_{\mathbb{R}^n} \prod_{i=1}^n \bigg[a(\mathbf{t}_i,\mathbf{x}_i)\varrho( \mathbf{t}_i-\mathbf{t}_{i-1}, \mathbf{x}_i-\mathbf{x}_{i-1})W(d\mathbf{t}_i d\mathbf{x}_i)\bigg], 
		\end{equation}
				where, $W$ is the white noise based on the Lebesque measure on $[0,1]\times \mathbb{R}$, $\mathbf{x}_0=0$,  $\varrho$ is the standard Gaussian heat kernel $$\varrho(t,x)= \frac{ e^{-x^2/2t}}{ \sqrt{2\pi t}},$$ and  $\Delta_n$ is the $n$\textbf{-dimensional simplex}
		\begin{equation}
			\label{equation: dimensional simplex} 
			\Delta_n := \{ \mathbf{t} \in [0,1]^n : 0=\mathbf{t}_0 <\mathbf{t}_1<\mathbf{t}_2<\cdots < \mathbf{t}_n \le 1 \}.
		\end{equation}
\end{itemize}	
\end{theorem}
We refer to \cite[Chapter 1]{Walsh1986} for an introduction on white noise (see also Section \ref{subsection: Wiener chaos} of this article).\\
As mentioned earlier, we will prove this theorem by investigating the connection between the collision measures of random walks with a model in statistical mechanics; hence, we will introduce many auxiliary notions in our paper, such as \textit{random environment} $\omega$,  \textit{partition function} $\mathfrak{Z}_N$ and $U$\textit{-statistics}  $\mathcal{S}^N_n( \cdot)$.\\
The general idea is that by associating each point $(n,z)$ on the grid $\mathbb{N} \times \mathbb{Z}$ with a random variable, we can change the underlying framework from studying a deterministic grid to studying a collection of random variables indexed by $\mathbb{N}\times \mathbb{Z}$. For a such collection, the range of possible tools from statistical mechanics is large. Indeed, the partition function we will use is developped to study a directed polymer model \cite{TomAlberts2014, Berger2021}.\\
The organization of this paper is as follows: Section \ref{section: Proof ideas} introduces some basic notions that we will use in the sequel to explain our main ideas, especially, the relation between the concept of partition functions and the collision measures $\Pi_N$. Section \ref{section: U-statistics} gives a brief review on $U$-statistics and Wiener chaos. At the end of this section, we prove some new theorems on the convergence of $U$-statistics, on which our asymptotic result on partition functions (Theorem \ref{theorem: convergence of parition functions}) is based. Section \ref{section:Limit theorems for partition functions} presents a short study on the random variable $\mathcal{Z}_a$ defined in \eqref{equation: Z_a}, and our proof of Theorem \ref{theorem: convergence of parition functions}. Finally, Section \ref{section: Limit theorems for collision events} combines all proved results to show the \textit{weak tightness} of $( \frac{1}{\sqrt{N}} \Pi_N, N \in \mathbb{N})$, and prove Theorem \ref{theorem: main theorem}.\\
Some auxiliary results are presented in the Appendix at the end of this article.
\section{Partition functions and main ideas of proof}
 \label{section: Proof ideas}
\subsection{Partition functions}
We introduce a collection $\omega:= \left( \omega(i,z) : i \in \mathbb{N}, z \in \mathbb{Z} \right)$ of independent Rademacher variables indexed by $\mathbb{N}\times\mathbb{Z}$, i.e., for all $(n,z) \in \mathbb{N}\times\mathbb{Z}$, $$\mathbf{P}( \omega(n,z)=-1) = \mathbf{P}( \omega(n,z)=1 ) =\frac{1}{2}.$$
These random variables are created by extending our existing probability space $(\Omega, \mathcal{A}, \mathbf{P})$ so that $\omega, S^{(1)}, S^{(2)}, \dots,S^{(k)}$ are independent.\\
In the sequel, for a real number $\beta$ and a real function $A$ on $\mathbb{N}\times \mathbb{Z}$, $\beta \omega$ and $A\omega$ are defined as:
\begin{align*}
	A\omega&:= \left( A(i,z)\omega(i,z) : i \in \mathbb{N}, z \in \mathbb{Z} \right).
	\\
	\beta \omega&:=  \left( \beta \omega(i,z) : i \in \mathbb{N}, z \in \mathbb{Z} \right).
\end{align*}
As briefly explained in Section \ref{section: Introduction}, the role $\omega$ is to add new degrees of freedom to the existing model, by which we have more flexibility to create more objects. The partition function $\mathfrak{Z}$ is one of such objects:
\begin{definition}
	\label{definition: partition function Z}
For any positive integer $N$ and any real function $A$ on $\mathbb{N}\times \mathbb{Z}$,  the \textbf{partition function} $\mathfrak{Z}_N( A )$ is defined as the conditional expectation:
\begin{align*}
	\mathfrak{Z}_N(A) := \mathbf{E} \left[ \text{ }\prod_{n=1}^N (1+A(n,S^{(1)}_n)\omega(n,S^{(1)}_n))  \text{ }\bigg| \text{   } \omega\text{ } \space \right].
\end{align*}
\end{definition}
Note that $\mathfrak{Z}_N(A)$ is a random variable depending only on the value of $\omega$.\\
\subsection{Main ideas}
The starting point of our paper and the proof of our main results is a heuristic relation between the partition functions $\mathfrak{Z}_N$  and the random measure $\Pi_N$:\\
Given a nonnegative bounded function $A$ on $\mathbb{N}\times \mathbb{Z}$, since $S^{(1)},\cdots,S^{(k)}$  are i.i.d.,  we have:
	\begin{align*}
	&\mathbf{E}\left[ \mathfrak{Z}_N\left( \frac{1}{N^{1/4}} A \right)^k \right] = \mathbf{E}\left[ \mathbf{E} \left[ \prod_{n=1}^N (1+ \frac{1}{N^{1/4}} A (n,S^{(1)}_n)\omega(n,S^{(1)}_n)) \bigg| \omega \right]^k \right] 	
	\\
	=& \mathbf{E} \Bigg[ \mathbf{E} \left[  \prod_{i=1}^k \prod_{n=1}^N (1+ \frac{1}{N^{1/4}} A (n,S^{(i)}_n)\omega(n,S^{(i)}_n)) \bigg| \omega \right] \Bigg]
	\\
	=& \mathbf{E}  \left[  \prod_{n=1}^N \prod_{i=1}^k (1+ \frac{1}{N^{1/4}} A (n,S^{(i)}_n)\omega(n,S^{(i)}_n)) \right] 
	\\
	= & \mathbf{E} \left[ \mathbf{E}  \left[  \prod_{n=1}^N \prod_{i=1}^k (1+ \frac{1}{N^{1/4}} A (n,S^{(i)}_n)\omega(n,S^{(i)}_n)) \bigg| S^{(1)},\cdots,S^{(k)}  \right] \right]
	\\
	= & \mathbf{E} \left[ \prod_{n=1}^N \mathbf{E}  \left[   \prod_{i=1}^k (1+ \frac{1}{N^{1/4}} A (n,S^{(i)}_n)\omega(n,S^{(i)}_n)) \bigg| S^{(1)},\cdots,S^{(k)}  \right] \right]
	\\
	= &  \mathbf{E}\bigg[ \prod_{n=1}^N  \bigg[ 1+ \frac{1}{N^{1/2}} \bigg(\sum_{ \substack{1 \le i < j \le k \space:\\ S_n^{(i)}= S_n^{(j)}=z}} A^2(n,z) \bigg)+ \frac{1}{N^{3/4}} (\dots) +...\bigg]\bigg]
\end{align*}
Then since $1+x \approx e^x$, heuristically, we deduce:
\begin{align*}
	\mathbf{E}&\left[ \mathfrak{Z}_N\left( \frac{1}{N^{1/4}} A \right)^k \right] \approx  \mathbf{E}\bigg[ \prod_{n=1 }^N \bigg[ \exp\big( \frac{1}{N^{1/2}} \sum_{ \substack{1 \le i < j \le k \space:\\ S_n^{(i)}= S_n^{(j)}=z}} A^2(n,z) \big)\bigg] \bigg]
	\\
	= & \mathbf{E} \bigg[  \exp \bigg( \frac{1}{N^{1/2}} \sum_{n=1}^N \sum_{ \substack{1 \le i < j \le k \space:\\ S_n^{(i)}= S_n^{(j)}=z}} A^2(n,z) \bigg)\bigg]
	= \mathbf{E} \bigg[  \exp\bigg(  \frac{1}{N^{1/2}} \Pi_N( f_N) \bigg)\bigg],
\end{align*}
where $f_N$ is a measurable function such that $f_N\left( \frac{n}{N}, \frac{z}{\sqrt{N}} \right)= A(n,z)$ for all $n \in \mathbb{N}, z \in \mathbb{Z}$.\\
In short, by abuse of notation, the above observation suggests that:
\begin{equation}
		\label{equation: heuristic calculation}
		\mathbf{E}\left[e^{ N^{-1/2}\Pi_N}\right]\approx \mathbf{E}\left[ (\mathfrak{Z}_N)^k \right].
\end{equation} 
In other words, if we have a good understanding of $\mathfrak{Z}$, we will have good information on  $\Pi_N, \Pi'_N$. \\
Then to study the partition function $\mathfrak{Z}_N$, we base our study on the paper \cite{TomAlberts2014}, in which Alberts et al. studied the scaling limit of $\mathfrak{Z}$ when the function $A$ is constant. In our study, we generalize their results for a sequence of functions $(A_N, N \in \mathbb{N})$ satisfying certain conditions. An expansion of Wiener chaoses emerges naturally in our limit objects because, as we will see, each term in the algebraic expansion of $\mathfrak{Z}_N$ (cf. Proposition \ref{proposition: relation between partition functions and U-statistics}) converges to a Wiener chaos. To this aim, we will have to introduce some $U$-statistics and study their asymptotic behavior in Section \ref{section: U-statistics}.
\subsection{Main results on partition functions}
We terminate this section by presenting our results on the asymptotic behavior of $\mathfrak{Z}_N$. The proofs will be presented later in Section \ref{section:Limit theorems for partition functions}.
\begin{notation} For $(t,x) \in [0,1]\times \mathbb{R}$, $[t,x]_N$ denotes the unique pair of integer $(i,z)$ such that :
	\begin{itemize}
		\item $(t,x) \in \left( \frac{i-1}{N},\frac{i}{N}\right]\times \left( \frac{z-1}{\sqrt{N}},\frac{z+1}{\sqrt{N}}\right]$,
		\item $i$ and $z$ have same parity.
	\end{itemize}
	\label{notation: (t,x)}
\end{notation} 

\begin{theorem}
	\label{theorem: convergence of parition functions}
	Let $(A_n, n \in \mathbb{N})$ be a sequence of real functions whose domain is $\mathbb{N} \times \mathbb{Z}$ such that:
	\begin{itemize}
		\item[i.]  $\sup_{N} \|  A_N\|_{\infty}  <+\infty$,
		\item[ii.]  there is a measurable function $a \in L^{\infty}( [0,1]\times \mathbb{R})$ such that:
		$$\lim_{N \rightarrow +\infty} A_N( [t,x]_N)= a(t,x) \quad \text{a.e.}$$
		
	\end{itemize}
	Then as $N$ converges to infinity, we have:
	$$\mathfrak{Z}_N(N^{-1/4} A_N ) \xrightarrow{\text{(d)}} \mathcal{Z}_{\sqrt{2}a}$$

\end{theorem}
This theorem is a generalization of Proposition 5.3 in \cite{TomAlberts2014} where the sequence $(A_N;N \in \mathbb{N})$ is replaced by a fixed constant $\beta \ge 0$.\\
We will also prove that under some conditions,  the partition functions are uniformly bounded in $L^k$:
\begin{theorem} 
	\label{theorem:uniform integrability of Z_k}
	For a sequence of real functions $(A_n, n \in \mathbb{N})$ on $\mathbb{N}\times\mathbb{Z}$ such that $ \sup_{N} \| A_N\|_{\infty} <+\infty$, we have:
	$$\limsup_{N} \mathbf{E}\left[ \left(\mathfrak{Z}_N\left( N^{-1/4} A_N  \right) \right)^k \right]  < +\infty .$$	
\end{theorem}
Notice that even though $k$ is fixed in our study, the definition of $\mathfrak{Z}$ does not depend on $k$. So, the above sequence $\left(\mathfrak{Z}_N\left( N^{-1/4} A_N  \right) , N \in \mathbb{N}\right)$ is also uniformly bounded in $L^p$ for all $p \in \mathbb{N}$, which implies directly the following corollary:
\begin{corollary}
	\label{corollary: L^k integrability}
For a sequence of real functions $A_1,A_2,...$ on $\mathbb{N}\times\mathbb{Z}$ such that $ \sup_{N} \| A_N\|_{\infty} <+\infty$, the sequence of random variables $\left(\mathfrak{Z}_N\left( N^{-1/4} A_N  \right) , N \in \mathbb{N}\right)$ is uniformly $L^k$-integrable.
\end{corollary}

\section{$U$-Statistics: related notions and limit theorems }
\label{section: U-statistics}

Let $E^N_n:=\{\mathbf{i} \in [\![1,N]\!]^n : \mathbf{i}_j \ne \mathbf{i}_l \text{ for } j \ne l\}$.
In this paper, we are interested in sums of the form:
\begin{equation}
	\sum_{ \mathbf{i} \in E^N_n} \sum_{\substack{ \mathbf{z} \in \mathbb{Z}^n\\ \mathbf{i} \leftrightarrow \mathbf{z}}} \overline{g_N}(\mathbf{i},\mathbf{z}) A_N(\mathbf{i},\mathbf{z}) \omega(\mathbf{i},\mathbf{z}) ,
	\label{equation: the sums}
\end{equation}

for some weight functions $\overline{g_N}$ specified later. The notation $\mathbf{i} \leftrightarrow \mathbf{z}$ means that:
\begin{notation}
	$\mathbf{i} \leftrightarrow \mathbf{z}$ means that for all $j \in [\![1,n]\!]$, the corresponding $j$-th coordinates of $\mathbf{i}$ and $\mathbf{z}$, namely $\mathbf{i}_j$ and $\mathbf{z}_j$, have the same parity. 
\end{notation} 
We will see that sums of this type appear naturally when we expand the partition functions $\mathfrak{Z}_N$ (see \eqref{equation: expansion of Z_N}).\\
The organization of this section is as follows: Sections \ref{subsection: U-statistics for space-time random environements} and \ref{subsection: Wiener chaos} introduce the framework of Theorem \ref{theorem: discrete Chaos to Wiener chaos} which is our main result on the convergence of sums of the form \eqref{equation: the sums}. Section \ref{subsection: Limit theorems for $U$-statistics} presents the proof of Theorem  \ref{theorem: discrete Chaos to Wiener chaos} and other related results.\\
The approach we used in this section is standard in the theory of $U$-statistics. Interested readers can consult the book \cite{Korolyuk1994} of Korolyuk and Borovskikh for a more rigourous introduction of this theory.
\subsection{Introduction of $U$-statistics $S^N_n$}\label{subsection: U-statistics for space-time random environements}

We first make precise the definition of the weight functions $(\overline{g_N}, N \in \mathbb{N})$ in the above sum.\\
Let $g$ be a function  in $L^2([0,1]^n \times \mathbb{R}^n)$. For each $N$, the weight functions $\overline{g_N}$ associated to $g$ is defined by the following procedure:\\
First, we partition the space $(0,1]^n \times \mathbb{R}^n$  in rectangles of the form:
$$\mathcal{R}^N_n:= \left\{\left( \frac{\mathbf{i}-\mathbf{1}}{N},\frac{\mathbf{i}}{N}\right]\times \left( \frac{\mathbf{z}-1}{\sqrt{N}},\frac{\mathbf{z}+\mathbf{1}}{\sqrt{N}}\right] : \mathbf{i} \in D^N_n, \mathbf{z} \in \mathbb{Z}^n, \mathbf{i} \leftrightarrow \mathbf{z} \right\} .$$
with $\mathbf{1}$ being the vector of ones and  $D^N_n$ being the \textit{integer simplex}:
\begin{equation}
	\label{equation: integer simplex}
	D^N_n := \{ \mathbf{i} \in [\![1,N]\!]^n : 1 \le \mathbf{i}_1 <\mathbf{i}_2\dots <\mathbf{i}_n \le N \}.
\end{equation}
Visually, $\mathcal{R}^N_n$ is a collection of nonoverlapping translations of the base rectangle:
$$\left( \frac{0}{N}, \frac{1}{N} \right]^n \times \left( \frac{-1}{\sqrt{N}} , \frac{1}{\sqrt{N}} \right]^n.$$
Then, the function $\overline{g_N}$ is defined as the average of $g$ on each rectangles above.\\
More precisely, for any $(\mathbf{t},\mathbf{x}) \in (0,1]^n\times \mathbb{R}^n$, $\overline{g_N}(\mathbf{t},\mathbf{x})$ is defined as the mean:
$$\overline{g_N}( \mathbf{t}, \mathbf{x}) := \frac{1}{|R|} \int_R g(\mathbf{s},\mathbf{y}) d\mathbf{s}\,d\mathbf{y} ,$$
where $R$ is the unique rectangle in $\mathcal{R}^N_n$ that contains $(\mathbf{t},\mathbf{x})$, and $|R|$ denotes the Lebesque measure of $R$. In probabilistic terms, $\overline{g_N}$ is simply the conditional expectation of $g$ onto the rectangles of $\mathcal{R}^N_n$. We note that $|R|= 2^nN^{-3n/2}$. This term will appear recurrently in most of our computations.\\
Suppose $(A_N, N \in \mathbb{N})$ is a sequence of real-valued functions on $\mathbb{N}\times \mathbb{Z}$.
\begin{notation}
	For any $n$-tuple $\mathbf{i} \in E^N_n$ and $n$-tuple $\mathbf{z} \in \mathbb{Z}^n$,  $A_N( \mathbf{i}, \mathbf{z})$ and $\omega( \mathbf{i}, \mathbf{z})$ denote
	\begin{align*}
		A_N( \mathbf{i}, \mathbf{z} )&:= A_N( \mathbf{i}_1, \mathbf{z}_1 )A_N( \mathbf{i}_2, \mathbf{z}_2 )...A_N( \mathbf{i}_n, \mathbf{z}_n ) ,\\
		\omega( \mathbf{i}, \mathbf{z} )&:= \omega( \mathbf{i}_1, \mathbf{z}_1 )\omega( \mathbf{i}_2, \mathbf{z}_2)...\omega( \mathbf{i}_n, \mathbf{z}_n ) ,
	\end{align*}
	with $\mathbf{i}_j$ being the $j$-th coordinate of $\mathbf{i}$ as defined previously.
\end{notation}
Now, we define the weighted $U$-statistics $\mathcal{S}^N_n$.
\begin{definition}
	\label{definition: weighted U-statistics}
	Suppose $(A_n, n \in \mathbb{N})$ is a sequence of bounded real-valued functions on $\mathbb{N}\times \mathbb{Z}$. For any function $g \in L^2([0,1]^n \times \Bbb{R}^n)$, the $U$- statistics $\mathcal{S}^{N}_n$ is defined as:
	\begin{equation}
		\label{equation: weighted U-statistics}
\mathcal{S}^N_{n}(g) := 2^{n/2} \sum_{\mathbf{i} \in E^N_n} \sum_{\substack{\mathbf{z} \in \mathbb{Z}^n :\\ \mathbf{i} \leftrightarrow \mathbf{z}}} \overline{g_N}\left( \frac{\mathbf{i}}{N}, \frac{\mathbf{z}}{\sqrt{N}} \right)A_N( \mathbf{i},\mathbf{z}) \omega(\mathbf{i},\mathbf{z}).
	\end{equation}
\end{definition}
We first give some basic properties of the $U$-statistics $S^N_n$.
\begin{proposition}
	\label{theorem: properties of U-statistics}
	Suppose $(A_N, N \in \mathbb{N})$ is a sequence of bounded real-valued functions on $\mathbb{N}\times \mathbb{Z}$. For all positive integers $n$ and $N$, we have:

	\begin{itemize}
		\item[i.] (Well-posedness) $\mathcal{S}^N_n(g)$ is well-defined and has zero mean for all $g \in L^2([0,1]^n \times \mathbb{R}^n)$.
		\item[ii.] (Linearity)  For all $f,g \in  L^2([0,1]^n \times \mathbb{R}^n)$, $\alpha,\beta  \in \mathbb{R}$
		$$\mathcal{S}^{N}_n(\alpha f+\beta g)= \alpha \mathcal{S}^{N}_n(f)+\beta \mathcal{S}^N_n(g).$$
		\item[iii.] ($L^2$-boundedness) If $c>0$ is a number such that such that $\|A_N\|_{\infty} \le c$, then for all $g \in L^2([0,1]^n\times \mathbb{R}^n)$ :
		$$\mathbf{E}[\mathcal{S}^N_n(g)^2] \le  c^{2n}N^{3n/2}\| g\|^2_{2}.$$
		\item[iv.]
		 (Uncorrelatedness of $U$-statistics of different orders) If $n_1,n_2$ are two different positive integers, then $\forall g_i \in L^2( [0,1]^{n_i}\times \mathbb{R}^{n_i}) \quad i=1,2,$
		$$\mathbf{E}[\mathcal{S}_{n_1}^N(g_1)\mathcal{S}_{n_2}^N(g_2)] =0.$$
	\end{itemize}
\end{proposition}
\begin{proof}
	Assume that $f$ and $g$ have compact supports, then the sums in $S^N_n(f)$ and $S^N_n(g)$ have a finite number of terms; thus, point $i$ is trivial. Point $ii$ is also trivial by recalling that  $\omega$ is a collection of centered random variables. Now, for point $iii$, observe that for any $\mathbf{i}, \mathbf{i}' \in E^N_{n}, \mathbf{x},\mathbf{x}' \in \mathbb{Z}^{n}$:
	$$\mathbf{E} \bigg[  \prod_{l=1}^{n} \omega( \mathbf{i}_l, \mathbf{x}_l)\prod_{l=1}^{n} \omega( \mathbf{i}'_l, \mathbf{x}'_l)\bigg]= \mathbb{1}_{ \{\mathbf{i} = \mathbf{i}' , \mathbf{x}=\mathbf{x}' \} } .$$
	Hence
	\begin{align*}
		\mathbf{E}\bigg[ \mathcal{S}^N_n(g)^2\bigg] & = 2^n \sum_{\mathbf{i} \in E^N_n} \sum_{\substack{\mathbf{z} \in \mathbb{Z}^n\\ \mathbf{i} \leftrightarrow \mathbf{z}}}   A_N( \mathbf{i},\mathbf{z})^2 	  \overline{g_N}\left( \frac{\mathbf{i}}{N}, \frac{\mathbf{z}}{\sqrt{N}} \right)^2      
		\\
		&\le 2^n \sum_{\mathbf{i} \in [\![ 1,N]\!]^n} \sum_{\substack{\mathbf{z} \in \mathbb{Z}^n\\ \mathbf{i} \leftrightarrow \mathbf{z}}}   c^{2n}   \frac{1}{|\mathcal{R}|} \int_{\mathcal{R}}g( \mathbf{t}, \mathbf{y})^2 d\mathbf{t}d\mathbf{y}
		\\
		& =N^{3n/2}c^{2n} \int_{[0,1]^n} \int_{\mathbb{R}^n} g(\mathbf{t},\mathbf{y})^2d\mathbf{t} \mathbf{y}.
	\end{align*}
	The last inequality is simply an application of the Cauchy-Schwarz lemma.\\
	So our theorem is valid for compactly supported functions. In other words, $g \mapsto S^N_n(g)$ is a linear Lipschitz continuous mapping that maps the space $L^2_{\text{compact}}([0,1]^n \times \mathbb{R}^n)$ into $L^2(\mathbf{P})$. Hence, all the properties $i,ii,iii$ can be extended naturally to all $L^2([0,1]^n \times \mathbb{R}^n)$ by the density of $L^2_{\text{compact}}([0,1]^n \times \mathbb{R}^n)$  in $L^2([0,1]^n \times \mathbb{R}^n)$.\\
	For the covariance relation in point $iv$, one can observe that if $\mathbf{i} \in E^N_{n_1}, \mathbf{x} \in \mathbb{Z}^{n_1},\mathbf{i}' \in E^N_{n_2}, \mathbf{x}' \in \mathbb{Z}^{n_2}$, then
	$$\mathbf{E} \bigg[  \prod_{l=1}^{n_1} \omega( \mathbf{i}_l, \mathbf{x}_l)\prod_{l=1}^{n_2} \omega( \mathbf{i}'_l, \mathbf{x}'_l)\bigg]= 0 .$$
	because there is necessarily one $\omega$ term that is distinct from the others, and its independence from the rest implies zero expectation.\\
	Hence, $iv$ is clearly true if $g_1,g_2$ have compact supports. The extension to non-compactly-supported functions can also be obtained by a density argument as above.
\end{proof}
Now, to characterize rigorously the limit of the $U$-statistics $(S^N_n, N \ge 1)$, we need to introduce the Wiener chaos. 
\subsection{Wiener chaos}
\label{subsection: Wiener chaos}
\subsubsection{White noise and stochastic integration on $[0,1]\times \mathbb{R}$}
\label{subsubsection: white noise}
This section recalls the elementary theory of white noise and stochastic integration on the measure space $([0,1]\times \mathbb{R}, \mathcal{B}, dt\otimes dx)$. Here $\mathcal{B}$ is the Borel $\sigma$-algebra, and $dt\otimes dx$ denotes Lebesque measure on $[0,1]\times \mathbb{R}$. For more details on Wiener chaos, we invite readers to read \cite[Chapter 1]{Nualart2006} or \cite[Chapter 11]{Kallenberg1997}.\\
Let $\mathcal{B}_f$ be the collection of all Borel sets of $[0,1]\times \mathbb{R}$ with finite Lebesgue measure. Observe that $\mathcal{B}= \sigma( \mathcal{B}_f)$.
\begin{definition}
A \textbf{white noise }on $[0,1]\times \mathbb{R}$ is a collection of mean zero Gaussian random variables indexed by $\mathcal{B}_f$
$$W= \{  W(A): A \in \mathcal{B}_f\}$$
such that for any $h \in \mathbb{N}$ and every finite collection $(A_1,A_2,\dots,A_h) $ of elements of $\mathcal{B}_f$, the tuple $(W(A_1),\dots, W(A_h))$ is a $h$-dimensional Gaussian vector, with mean zero and covariance structure: 
$$\mathbf{E}[ W(A)W(B)]= | A \cap B |.$$
\end{definition}
So in particular, if $A$ and $B$ are disjoint then $W(A)$ and $W(B)$ are independent.\\
For any $g \in L^2( [0,1] \times \mathbb{R}, \mathcal{B}, dt\,dx)$, the stochastic integral 
$$I_1 (g) := \int_0^1 \int_{\mathbb{R}} g(t,x)W(dt\,dx)$$
is constructed by first defining $I_1$ on simple functions then extending $I_1$ via density arguments \cite[p.210]{Kallenberg1997}. In the end, for  each $g \in L^2( [0,1]\times \mathbb{R})$, we have that $I_1(g) \sim N( 0, \| g\|_{2}^2)$, so in particular, $I_1$ preserves the Hilbert space structure of $L^2( [0,1]\times \mathbb{R})$,
$$\mathbf{E}( I_1(g)I_1(h))= \int_{0}^1 \int g(t,x)h(t,x)dt\,dx$$
This construction idea can be extended to higher dimensions (see  \cite[p. 9,10]{Nualart2006}) to give a sense of the following notation of multiple stochastic integrals for any $n>1$ and function $g \in L^2([0,1]^n \times \mathbb{R}^n)$: 
$$I_n(g):= \int_{[0,1]^n } \int_{\mathbb{R}^n} g( \mathbf{t}, \mathbf{x}) W^{\otimes n}(d\mathbf{t} d\mathbf{x}),$$
where $W^{\otimes n}(d\mathbf{t} d\mathbf{x}):=W(d\mathbf{t}_1 d\mathbf{x}_1) W(d\mathbf{t}_2 d\mathbf{x}_2)\cdots W(d\mathbf{t}_n d\mathbf{x}_n)$.\\  
However, the mapping $I_n: L^2([0,1]^n\times \mathbb{R}^n) \rightarrow L^2( \mathbf{P})$ is no longer injective. For example, if $A,B$ are two disjoint compacts of $[0,1]\times \mathbb{R}$, we see that $I_2( \mathbb{1}_{A \times B}) = W(A)W(B)= I_2( \mathbb{1}_{B \times A})$ even though $\mathbb{1}_{A\times B} \ne  \mathbb{1}_{B \times A}$. Nonetheless, we observe that the restriction of $I_n$ on the subspace $L_{\text{sym}}^2([0,1]^n \times \mathbb{R}^n) $ (see Definition \ref{definition: symmetry}) of $L^2([0,1]^n \times \mathbb{R}^n) $ is an isometry \cite[p. 9,10]{Nualart2006}.
\begin{definition}
	\label{definition: symmetry}
	A function $g \in L^2( [0,1]^n \times \mathbb{R}^n)$ is said to be symmetric if $g( \mathbf{t}, \mathbf{x})= g(  \pi \mathbf{t}, \pi \mathbf{x} )$ for all  $(\mathbf{t}, \mathbf{x})\in[0,1]^n \times \mathbb{R}^n$ , permutation $\pi$ on  $\{1, \hdots, n\}$, where $\pi \mathbf{t}  :=	\mathbf{t}_{\pi(1)}, \hdots,\mathbf{t}_{\pi(n)}, \pi \mathbf{x} :=	 \mathbf{x}_{\pi(1)}, \hdots,\mathbf{x}_{\pi(n)}$.\\
The set $L^2_{\text{sym}}([0,1]^n \times \mathbb{R}^n)$ is then defined as the subspace of all symmetric functions of $L^2([0,1]^n \times \mathbb{R}^n)$.
\end{definition}
\begin{notation} For any $n \in \mathbb{N}$, $g \in L^2 ([0,1]^n \times \mathbb{R}^n)$,  we denote by $\text{Sym}\,g(\mathbf{t},\mathbf{x})$ the symmetrization of $g$ defined by:
	\begin{equation}
		\text{Sym}\,g(\mathbf{t},\mathbf{x}):= \frac{1}{n!} \sum_{\pi \text{ is a permutation of }\{1,\dots,n\}} g( \pi \mathbf{t}, \pi \mathbf{x}).
	\end{equation}
\end{notation} 
In summary, we have the following theorem which is a standard result in the theory of stochastic integration:
\begin{theorem}
		\label{theorem: Isometry and Stochastic Integration}
	There exists a continuous linear mapping $I_n: L^2 ([0,1]^n\times \mathbb{R}^n) \rightarrow L^2(\mathbf{P})$ such that for any $n$-tuple of disjoint finite measurable sets $A_1,A_2,\cdots A_n$ in $\mathcal{B}([0,1]\times \mathbb{R})$:
	$$I_n(\mathbb{1}_{ A_1 \times A_2 \cdots \times A_n})= W(A_1)W(A_2)\cdots W(A_n)$$
	Furthermore, for all  $g \in L^2([0,1]^n \times \mathbb{R}^n)$, \begin{equation}
		\mathbf{E}\left[I_n(g)^2\right] \le \| g\|^2_2,
	\end{equation}
	and the equality occurs if and only if $g$ is symmetric.
\end{theorem}
\begin{proof}
	The first part is a summary of the construction of multiple stochastic integration in \cite[p.8,9]{Nualart2006}. For the inequality, observe that:
	\begin{align*}
		\mathbf{E}&\left[I_n(g)^2\right] = \mathbf{E}\left[I_n(\text{Sym }g)^2\right] \stackrel{\text{Isometry}}{=} \|\text{Sym }(g)\|_2^2 \\
		& =\left\|  \frac{1}{n!}  \sum_{\pi \text{ is a permutation}} g\circ \pi \right\|^2_2  \\
		& \le  \frac{1}{n!} \sum_{\pi \text{ is a permutation}} \| g\circ \pi\|_2^2= \|g\|^2_2 
	\end{align*}
by Cauchy-Schwarz's inequality. Here, by abuse of notation, $(g\circ \pi ) ( \mathbf{t},\mathbf{x})$ denotes $g(\pi \mathbf{t},\pi \mathbf{x})$.
\end{proof}
\subsubsection{Wiener chaos on $[0,1]\times \mathbb{R}$}
This section provides a short introduction to the Wiener chaos's theory.\\
Wiener chaos may be regarded as a way of representing random variables as infinite sums of multiple stochastic integrals.\\
For a white noise $W$, we denote by $\mathcal{F}_W$ the complete $\sigma$-algebra generated by random variables $(W(A), A \in \mathcal{B}_f)$, the Wiener chaos decomposition theorem states (see \cite[Theorem 1.1.2]{Nualart2006}):
\begin{proposition}(Wiener chaos decomposition)\label{proposition: Wiener chaos decomposition}
For every random variable $X \in L^2 (\Omega,  \mathcal{F}_W, \mathbf{P})$, there is a unique sequence of symmetric functions $g_n \in L^2_{\text{Sym}}( [0,1]^n \times \mathbb{R}^n),n \ge 1$, such that:
$$X=\sum_{n=0}^{\infty}I_n(g_n) .$$
Here $g_0$ is simply a constant and $I_0$ is the identity mapping on the constants.
\end{proposition}
In fact,  for $n \ge 1$, the terms of the chaos series are all mean zero, so $g_0$ must be the mean of $X$. Moreover, by the orthogonality of $I_{n_1}(g_1)$ and $I_{n_2}(g_2)$ for $n_1 \ne n_2$(see \cite[p.9]{Nualart2006}), we have the relation:
$$\mathbf{E}[X^2]= \sum_{n=0}^{\infty} \| g_n\|^2_{2}.$$
Now, we define two important spaces of collections of functions:
\begin{definition}
	\label{definition: Fock spaces}
	The Fock space over $L^2([0,1]\times \mathbb{R})$ is defined to be the Hilbert space:
	\begin{equation}
		F:= \left\{ \mathbf{g}=(g_0,g_1,\dots) \in \bigoplus_{n=0}^{\infty} L^2( [0,1]^n \times \mathbb{R}^n) : \sum_{n=0}^{\infty} \|g_n\|^2_2 <\infty \right\}
	\end{equation}
	equipped with the inner product $ \langle \mathbf{g},\mathbf{f} \rangle_F =\sum_{n=0}^{\infty} \langle \mathbf{g}_n,\mathbf{f}_n\rangle_{L^2([0,1]^n\times \mathbb{R}^n)}$.\\
	Then, the symmetric Fock space $F_{\text{sym}}$ is defined as the Hilbert subspace of $F$ that contains only collections of symmetric functions, i.e., $$F_{\text{sym}} := F \bigcap \left(\bigoplus_{n=0}^{\infty} L_{\text{sym}}^2( [0,1]^n \times \mathbb{R}^n)\right).$$
\end{definition}
The result in Proposition \ref{proposition: Wiener chaos decomposition} works also in reverse, that is, the mapping 
\begin{align*}
	\mathbf{I} \colon &F_{\text{sym}} \phantom{abcde} \xrightarrow{\phantom{abcdefg}}  L^2( \Omega, \mathcal{F}_W, \mathbf{P})\\
	 &(g_0,g_1,\dots) \xmapsto{\phantom{L^\infty(T)}} \sum_{n \ge 0}I_n(g_n)
\end{align*}
is an isometry. This fact will be useful for the justification for the well-posedness of $\mathcal{Z}_{a}$ in Section \ref{section:Limit theorems for partition functions}.

\subsection{Limit theorems for $U$-statistics}
\label{subsection: Limit theorems for $U$-statistics}
In this section, we prove two limit theorems (Theorem  \ref{theorem: convergence of U-statistics} and Theorem \ref{theorem: discrete Chaos to Wiener chaos}) for our $U$-statistics $S^N_n$ defined by \eqref{equation: weighted U-statistics}. They are extensions of Theorem 4.3  and Lemma 4.4  in \cite{TomAlberts2014} with non-constant $A_n$. The result of the second theorem will be useful for the rest of this paper while the first is crucial for the proof of the second. 
\begin{theorem}(Convergence of $U$-statistics to Stochastic Integrals)
	\newline 	\label{theorem: convergence of U-statistics}
	Suppose the functions $A_1,A_2,\dots$ in the definition \ref{definition: weighted U-statistics} of the $U$-statistics sastify the following conditions:
	\begin{itemize}
		\item[i.]  $\sup_{N} \|  A_N\|_{\infty}  <+\infty .$
		\item[ii.] There is a measurable function $a \in L^{\infty}( [0,1]\times \mathbb{R})$ such that:
		$$\lim_{N \rightarrow +\infty} A_N( [t,x]_N)= a(t,x) \quad \text{a.e.}$$
	\end{itemize}
	Then, for any positive integer $n$ and function $g \in L^2( [0,1]^n\times \mathbb{R}^n)$, we have:
	$$N^{-3n/4}\mathcal{S}^N_n(g) \xrightarrow[N \rightarrow+\infty]{\text{(d)}} \int_{[0,1]^n} \int_{ \Bbb{R}^n}g( \mathbf{t},\mathbf{x})a^{\otimes n}(\mathbf{t},\mathbf{x})W^{\otimes n}(d\mathbf{t}d\mathbf{x}) .$$
	Moreover, for any finite collection of $n_1,\hdots,n_m\in \mathbb{N}_0$ and $g_1,\hdots,g_m$ with $g_i \in L^2([0,1]^{n_i}\times\mathbb{R}^{n_i})$, one has the joint convergence
	$$(N^{-3n_1/4}\mathcal{S}^{N}_{n_1}(g_1),...,N^{-3n_1/4}\mathcal{S}^{N}_{n_m}(g_m)) \xrightarrow[N \rightarrow +\infty]{(d)} (\tilde{I}_{n_1}(g_1),...,(\tilde{I}_{n_m}(g_m)) .$$
	where, for $n \ge 1$,
	\begin{equation}
		\tilde{I}_n(g):=\int_{[0,1]^n} \int_{ \Bbb{R}^n}g( \mathbf{t},\mathbf{x})a^{\otimes n}(\mathbf{t},\mathbf{x})W^{\otimes n}(d\mathbf{t}d\mathbf{x}) ,
	\end{equation}
	and
	\begin{equation}
		\label{definition: a otimes}
		a^{\otimes n}(\mathbf{t},\mathbf{x}) = a( \mathbf{t}_1,\mathbf{x}_1) \cdots a( \mathbf{t}_n,\mathbf{x}_n).
	\end{equation}
\end{theorem}

\begin{theorem}	\label{theorem: discrete Chaos to Wiener chaos}
	Suppose the functions $A_1,A_2,\dots$ in the definition \ref{definition: weighted U-statistics} of the $U$-statistics satisfy the following conditions:
	\begin{itemize}
		\item[i.]  $\sup_{N} \|  A_N\|_{\infty}  <c$ for some $c>0$.
		\item[ii.] There is a measurable function $a \in L^{\infty}( [0,1]\times \mathbb{R})$ such that:
		$$\lim_{N \rightarrow +\infty} A_N( [t,x]_N)= a(t,x) \quad \text{a.e.}$$
	\end{itemize}
Then if $(g_n, n \in \mathbb{N}_0)$ is a sequence of functions such that $(c^ng_n, n \in \mathbb{N}_0)$ belongs to the Fock space $F$, we have:
$$ \sum_{n=0}^{\infty} N^{-3n/4} \mathcal{S}^N_n(g_n) \xrightarrow[N \rightarrow \infty]{(d)} \sum_{n=0}^{\infty} \int_{[0,1]^n \times \mathbb{R}^n} g_n( \mathbf{t}, \mathbf{x})a^{\otimes n}( \mathbf{t}, \mathbf{x}) W^{\otimes n}(d \mathbf{t} d \mathbf{x}).$$
\end{theorem}
\begin{remark}
	 		Notice that by the definition, $(\mathbf{i},\mathbf{z}) \mapsto A_n( \mathbf{i}, \mathbf{z})$ is symmetric, hence  $\mathcal{S}^N_n(g_n)= \mathcal{S}^N_n(\text{Sym}(g_n))$. This allows us to only consider symmetric functions in the proof of Theorem \ref{theorem: discrete Chaos to Wiener chaos}.
	 		\label{remark: symmetry of S}
\end{remark}

Besides, to prove the above two theorems, we will repeatedly use the following lemma in Billingsley \cite[Theorem 3.2]{Billingsley1999}:
\begin{lemma}
	\label{lemma: a standard result on weak convergence}
	For $n,N \in \mathbb{N}$, let $Y^N_n, Y_n, Y^N, Y$ be real-valued random variables defined on a common probability space such that $Y_n  \xrightarrow[n \rightarrow +\infty]{(d)} Y$ and that for all $n$, $Y^N_n \xrightarrow[N \rightarrow +\infty]{(d)} Y_n$. If for each $\epsilon>0$, $$\lim_{n \rightarrow \infty} \limsup_{N \rightarrow \infty} \mathbf{P}( |Y^N_n-Y^N| \ge \epsilon)=0$$
	Then $Y^{N} \xrightarrow[N\rightarrow \infty]{(d)} Y$.
\end{lemma}

\begin{proof}[Proof of Theorem \ref{theorem: convergence of U-statistics}] Let $c:= \sup_N \|A_N\|_{\infty}.$\\
	\textbf{Step 1:}  Let $n =1$ and assume that $g$ is a continuous and compactly supported function. 
	\newline
	Rewrite $N^{-3/4}\mathcal{S}^N_1(g)$ as a weighted sum of elements in $\omega$:
	\begin{equation}
		\label{equation: n=1, g continuous compactly supported}
			N^{-3/4}\mathcal{S}^N_1(g)= \sum_{i \in [\![1,N]\!]} \sum_{\substack{z \in \mathbb{Z}}} \mathbb{1}_{ i \leftrightarrow z}.2^{1/2}N^{-3/4}  \overline{g}_N	\left( \frac{{i}}{N}, \frac{{z}}{\sqrt{N}} \right)A_N( {i},{z}) \omega({i},{z}) .
	\end{equation}
	Because, $g$ has compact support, the number of nonzero terms in the above sum is finite. Besides, recall that $\omega$ is a collection of independent random variables having zero mean and  variance $1$, one has:
	\begin{align*}\mathbf{E}\left[N^{-3/2}\mathcal{S}^N_1(g)^2 \right]= 2N^{-3/2}\sum_{i \in [\![1,N]\!]} \sum_{\substack{z \in \mathbb{Z}\\ i \leftrightarrow x}} \overline{g_N}\left( \frac{{i}}{N}, \frac{{z}}{\sqrt{N}} \right)^2A_N( {i},{z})^2
		\\
		= \int_{[0,1]} \int_{\mathbb{R}} \overline{g_N}^2(t,x)A_N\left([(t,x)]_N\right)^2dt\,dx \xrightarrow[N \rightarrow +\infty]{} \int_{[0,1]}\int_{\mathbb{R}} g^2a^2(t,x)dt\,dx .
	\end{align*}
	where the last convergence follows from the dominated convergence theorem and the fact that $g$ is continuous and compactly supported.\\
	Hence, by Lindeberg-Feller's central limit theorem  (see \cite[Theorem 3.4.5]{Durrett2010}), 
	$$N^{-3/4}\mathcal{S}^N_1(g) \xrightarrow[N \rightarrow+\infty]{(d)}\mathcal{N}(0, \int_{[0,1]\times \mathbb{R}} g^2a^2(t,x)dt\,dx),$$
	Note that the second condition of Lindeberg-Feller is satisfied because the supremum of all the terms in Equation \eqref{equation: n=1, g continuous compactly supported} is smaller or equal to $$ 2^{1/2}N^{-3/4}\|g\|_{\infty} \|A_N\|_{\infty},$$ which converges to $0$ when $N \rightarrow \infty$. Thus, by the isometry of the stochastic integration $I_1$, we have proved that:
	$$ N^{-3/4}\mathcal{S}^N_1(g) \xrightarrow[N \rightarrow+\infty]{(d)} \int_{[0,1]}\int_{\mathbb{R}}g(t,x)a(t,x)W(dt\,dx)=\tilde{I}_1(g).$$
	\textbf{Step 2:} Let $n=1$ and $g$ be any function in $L^2([0,1]\times \mathbb{R})$.\\
	 Because the space of continuous and compactly supported functions $C_{c}([0,1]\times \mathbb{R})$ is dense in $L^2([0,1]\times \mathbb{R})$ \cite[Theorem 4.12]{Brezis2011}, there exists a sequence of continuous and compactly supported functions $(g_m, m \in \mathbb{N}_0)$ converging to $g$ in $L^2$. 
	Thus, by combining with the fact that $\|a\|_{\infty} \le c$, this implies $$\tilde{I_1}(g_m) \xrightarrow[m \rightarrow \infty]{L^2} \tilde{I_1}(g).$$
	Besides, for all $N$, observe that:
	$$\mathbf{E}\left[N^{-3/2} \left(\mathcal{S}_1^N(g-g_m)\right)^2 \right] \le c^2 \| g-g_m\|^2_{2}.$$
	So the above observations and Step 1 give the following diagram:	
	\begin{center}
			\begin{tikzcd}[row sep=huge, column sep = huge]
			N^{-3/4}\mathcal{S}^N_1(g_m) \arrow{r}{(d)}[swap]{N \rightarrow +\infty} \arrow{d}[swap]{\text{in } L^2(\mathbf{P}) \text{ , uniformly in }N}{m \rightarrow+\infty} & \tilde{I}_1(g_m) \arrow{d}{L^2}[swap]{m \rightarrow+\infty}
			\\
			N^{-3/4}\mathcal{S}^N_1(g) & \tilde{I}_1(g)	
		\end{tikzcd}
	\end{center}
	Thus thanks to Lemma \ref{lemma: a standard result on weak convergence}, we duce that 
	$N^{-3/4}\mathcal{S}^N_1(g) \xrightarrow[N \rightarrow+\infty]{(d)} \tilde{I_1}(g).$
	\\
	\textbf{Step 3:} Now we will prove that for all $m \in \mathbb{N}$ and $m$ functions $g_1,g_2,\dots,g_m \in L^2([0,1]\times \mathbb{R})$, we have:
	\begin{equation}
		\label{equation: joint convergence}
			\left( N^{-3/4}\mathcal{S}^N_1(g_1), \dots,N^{-3/4}\mathcal{S}^N_1(g_m) \right) \xrightarrow[N \rightarrow \infty]{(d)} ( \tilde{I}_1(g_1),\dots,\tilde{I}_1(g_m) ).
	\end{equation}
	Indeed, for any $m$ real numbers $\alpha_1,\dots,\alpha_m$, our result in \textit{Step 1} shows that:
	$$\alpha_1 N^{-3/4}\mathcal{S}^N_1(g_1)+\cdots  \alpha_m N^{-3/4}\mathcal{S}^N_1(g_m)  \xrightarrow[N \rightarrow \infty]{(d)}  \alpha_1 \tilde{I}_1(g_1)+\cdots \alpha_m \tilde{I}_1(g_m).$$
	Thus, by Cramer-Wold Theorem \cite[Corollary 4.5]{Kallenberg1997}, the convergence \eqref{equation: joint convergence} is valid.\\
	\textbf{Step 4:} For $n>1$ and  $g$ of the form:
	\begin{equation}
		\label{equation: product of disjoint functions}
		g(\mathbf{t},\mathbf{x})=  g^{(1)}(\mathbf{t}_1,\mathbf{x}_1)g^{(2)}(\mathbf{t}_2,\mathbf{x}_2)\cdots g^{(n)}( \mathbf{t}_n,\mathbf{x}_n).
	\end{equation}
	where $g^{(j)} $ are functions of $ L^2([0,1]\times \mathbb{R})$ with disjoint compact supports.
	\newline
	As the supports of $g^{(j)}$ are disjoint compacts, if $N$ is large enough, the supports of $\overline{g^{(j)}_N}$ are also disjoint.
	Thus we have the first equality in the following argument:
	\newline
	$$N^{-3n/4}\mathcal{S}^{N}_n(g)= \prod_{j=1}^nN^{-3/4}\mathcal{S}^{N}_1(g^{(j)}) \xrightarrow[N \rightarrow \infty]{(d)} \prod_{j=1}^n \tilde{I}_1(g^{(j)})=\tilde{I}_n(g).$$
	The latter limit in law is obtained by using the convergence of the joint random variables $(N^{-3/4}\mathcal{S}^{N}_1(g^{(1)}) ,...,N^{-3/4}\mathcal{S}^{N}_1(g^{(n)}) ).$\\
	\textbf{Step 5:} Now, for any $m$-tuple of functions $g_1,g_2,\hdots,g_m$ of the form \eqref{equation: product of disjoint functions}, by a similar argument, one can show the joint convergence:
	$$\left( N^{-3n/4}\mathcal{S}^N_n(g_1), \hdots,N^{-3n/4}\mathcal{S}^N_n(g_m) \right) \xrightarrow[N \rightarrow +\infty ]{(d)} (\tilde{I}_n(g_1),\hdots, \tilde{I}(g_m)).$$
	Hence, by the linearity of $\mathcal{S}$ and $I_n$, for any linear combination $g$ of functions of the form \eqref{equation: product of disjoint functions}, one has the convergence 
	$N^{-3n/4}\mathcal{S}^{N}_n(g) \xrightarrow[N \rightarrow +\infty ]{(d)} \tilde{I}_n(g).$\\
Besides, the space of such linear combinations is dense in $L^2([0,1]^n \times \mathbb{R}^n)$(the space of step functions is dense in $L^2([0,1]^n \times \mathbb{R}^n)$ \cite[Proof of Theorem 4.13]{Brezis2011}, then we shrink the support of each step function in an appropriate way), by a same density argument as in \textit{Step 1}, one can conclude that for any $g \in L^2([0,1]^n \times \mathbb{R}^n).$:
		$$N^{-3n/4}\mathcal{S}^{N}_n(g) \xrightarrow[N \rightarrow +\infty ]{(d)} \tilde{I}_n(g).$$
	The proof for the desired joint convergence for different $n$ is just a repeat of\textit{ Step 4} and \textit{Step 5}.
\end{proof}
\begin{proof}[Proof of Theorem \ref{theorem: discrete Chaos to Wiener chaos}]Without loss of generality, we assume $c=1$.\\
Using Cauchy-Schwarz's inequality as in the proof of Theorem \ref{theorem: Isometry and Stochastic Integration}, one can show that:
$$\text{Sym } \mathbf{g} := ( \text{Sym }g_0,  \text{Sym }g_1, \text{Sym } g_2,\hdots) \in F_{ \text{Sym}}.$$
By the symmetry of $\mathcal{S}^N_n$ and the stochastic integrations $(\tilde{I}_n, n \in \mathbb{N}_0$), without loss of generality, we can suppose that $g \in F_{\text{Sym}}$ (see Remark \ref{remark: symmetry of S}).\\
Since $\mathbf{I}$ is an isometry from the symmetric Fock space into $L^2(\mathbf{P})$, we automatically get that $\sum_{n=0}^M \tilde{I}_n(g_n) \longrightarrow \sum_{n \ge 0} \tilde{I_n}(g_n)$ in $L^2( \mathbf{P})$, as $M$ goes to infinity. Since $\text{Var}(N^{-3n/4}\mathcal{S}^N_n(g_n)) \le  \| g_n\|_2^2$ (see Proposition \ref{theorem: properties of U-statistics}), this also implies
$$\sum_{n=0}^M N^{-3n/4}\mathcal{S}^N_n(g_n)\xrightarrow[M\rightarrow+\infty ]{}\sum_{n=0}^{\infty} N^{-3n/4}\mathcal{S}^N_n(g_n)$$
in $L^2(\mathbf{P})$, uniformly in $N$. Theorem \ref{theorem: convergence of U-statistics} implies that:
	$$\sum_{n=0}^M N^{-3n/4} \mathcal{S}^N_n(g_n ) \xrightarrow[N \rightarrow \infty]{(d)} \sum_{n=0}^M \tilde{I}_n(g_n ).$$
Combining these asymptotic results, we have by Lemma \ref{lemma: a standard result on weak convergence} the following diagram:
\begin{center}
		\begin{tikzcd}[row sep=huge, column sep = huge]
		\sum_{n=0}^M N^{-3n/4}S^N_n(g_n)
		\arrow{r}{(d)}[swap]{N \rightarrow +\infty} \arrow{d}[swap]{\text{in }L^2,\text{ uniformly in }N}{M \rightarrow+\infty}
		& \sum_{n=0}^M \tilde{I}_n(g_n)
		\arrow{d}{\text{in }L^2}[swap]{M \rightarrow+\infty}
		\\
		\sum_{n=0}^{\infty} N^{-3n/4}S^N_n(g_n)
		\arrow{r}{(d)}[swap]{N \rightarrow +\infty}
		&
		 \sum_{n=0}^{\infty} \tilde{I}_n(g_n) .
	\end{tikzcd}
\end{center}
\end{proof}

\section{Limit theorems for paritition functions}
\label{section:Limit theorems for partition functions}
In this section, we study the convergence of partition functions $\mathfrak{Z}_N$ (Definition \ref{definition: partition function Z}). First, we verify the well-posedness of the limit value $\mathcal{Z}_a$ given in Theorem \ref{theorem: main theorem}, for all $a \in L^{\infty}([0,1]\times \mathbb{R})$ by
		$$\mathcal{Z}_{a} := 1+\sum_{n=1}^{\infty} \int_{\Delta_n} \int_{\mathbb{R}^n} \prod_{i=1}^n \bigg[a(\mathbf{t}_i,\mathbf{x}_i)\varrho( \mathbf{t}_i-\mathbf{t}_{i-1}, \mathbf{x}_i-\mathbf{x}_{i-1})W(d\mathbf{t}_i, d\mathbf{x}_i)\bigg].  $$
where $\varrho$ is the gaussian kernel, and $W$ is the white noise based on the Lebesque measure on $[0,1]\times \mathbb{R}$.
\subsection{Study of $\mathcal{Z}_{a}$}

\subsubsection{Brownian motion and simple random walk}
\label{subsection: Brownian motion and simple random walks}
Let $(S_n, n \in \mathbb{N}_0)$ denote  a simple random walk on $\mathbb{Z}$  and $(B_t, t \in \mathbb{R}_{\ge 0})$ denote a Brownian motion on $\mathbb{R}$.
\\
For $i \in \mathbb{N}$, $ t \ge 0$ and $x \in \mathbb{R}$, we define: 
\begin{equation}
		\label{equation: transitions}
		p(i,x) := \mathbf{P}(S_i=x) \qquad \varrho(t,x) := \frac{e^{-x^2/2t}}{\sqrt{2\pi t}} 
\end{equation}
We will make heavy use of the finite dimensional distributions of both simple random walk and Brownian motion. For notations, we introduce for $n \in \mathbb{N}$, $\mathbf{i} \in D^N_n$ ($D^N_n$ being the integer simplex \eqref{equation: integer simplex}), $\mathbf{z} \in \mathbb{Z}^n$, $\mathbf{t} \in \Delta_n$ ($\Delta_n$ being the real simplex \eqref{equation: dimensional simplex}), $\mathbf{x} \in \mathbb{R}^n$:
\begin{equation}
	p_n(\mathbf{i},\mathbf{z}) := \prod_{j=1}^n p(\mathbf{i}_j-\mathbf{i}_{j-1},\mathbf{z}_j-\mathbf{z}_{j-1})= \mathbf{P}(S_{\mathbf{i}_1}=\mathbf{z}_1 , \dots ,S_{\mathbf{i}_n}=\mathbf{z}_n  ) ,
\end{equation}
and
\begin{equation}
\varrho_n(\mathbf{t},\mathbf{x}) := \prod_{j=1}^n \varrho(\mathbf{t}_j-\mathbf{t}_{j-1},\mathbf{x}_j-\mathbf{x}_{j-1}).
\end{equation}
\\
For convenience, we  respectively extend the domains of $p_n$ and $\varrho$ to $[\![1,N]\!]^n \times \mathbb{Z}^n$ and to $[0,1]^n \times \mathbb{R}^n$ by letting $p_n$ and $\varrho_n$ to be zero outside $D^N_n \times \mathbb{Z}^n$ and $\Delta_n \times \mathbb{R}^n$.
\subsubsection{Wiener chaos for Brownian transition probabilites}
\label{subsubsection: Wiener chaos for Brownian transition}
The Brownian transition probabilites can generate many elements in the Fock space $F$ (see Definition \ref{definition: Fock spaces}). Let us recall here Notation \eqref{definition: a otimes} of $a^{\otimes}$.
\begin{proposition}
	\label{propostion: definition of varrho(a) }
	For every measurable bounded function $a \in L^{\infty}([0,1]\times \mathbb{R})$, let  $\boldsymbol{\varrho}( a ) :=  (1, a \varrho_1,a^{\otimes 2} \varrho_2,a^{\otimes 3} \varrho_3,\dots ). $ be a weighted ordered collection (indexed by $\mathbb{N}_0$) of Brownian transition probabilites $\varrho_n$ that depends on $a$. \\
	Then, $\boldsymbol{\varrho}(a)$ is an element in the Fock space $F$,i.e., $\sum_{n \ge 0} \| a^{\otimes n}\varrho_n\|^2_{2} < \infty.$
\end{proposition}
\begin{remark}
	In particular, if $a$ is a constant function, that is, $a$ is equal to some constant $\beta$ then $\boldsymbol{\varrho}(\beta)=  (1, \beta \varrho_1,\beta^2 \varrho_2,\beta^3 \varrho_3,\dots ) $.
\end{remark}

\begin{proof}[Proof of Proposition \ref{propostion: definition of varrho(a) }]
	Recall that $a$ is a bounded function, then there is a positive number $\beta$ such that $\|a\|_{\infty} \le \beta$. Hence, $\sum_{n \ge 0} \| a^{\otimes n}\varrho_n\|^2_{ 2} \le \sum_{n \ge 0} \beta^{2n} \| \varrho_n\|^2_{ 2} .$ Thus, it suffices to prove that $\boldsymbol{\varrho}(\beta)$ belongs to the Fock space.	Indeed, observe that when $\mathbf{t} \in \Delta_n$,
	$$\varrho_n( \mathbf{t},\mathbf{x})^2 = \varrho_n( \mathbf{t}, \mathbf{x}\sqrt{2})  \prod_{j=1}^{n} \frac{1}{\sqrt{2\pi(\mathbf{t}_j-\mathbf{t}_{j-1})}}.$$
	Hence,
\begin{align*}
	\int_{[0,1]^n}\int_{\mathbb{R}^n} \varrho_n( \mathbf{t},\mathbf{x})^2  d\mathbf{x} d\mathbf{t}
	&=	\int_{\Delta_n}\int_{\mathbb{R}^n} \varrho_n( \mathbf{t},\mathbf{x})^2  d\mathbf{x} d\mathbf{t}
	\\&=(4\pi)^{-n/2}\int_{\Delta_n} \prod_{j=1}^n \frac{1}{\sqrt{\mathbf{t}_j-\mathbf{t}_{j-1}}} d\mathbf{t}
	\\
	&= (4\pi)^{-n/2} B\left( \frac{1}{2}, \frac{1}{2},\cdots,\frac{1}{2},1\right)=\frac{1}{2^n \Gamma( (n/2)+1) }.
\end{align*}
where $B$ is the Beta function and $\Gamma$ is the Gamma function.\\
The second to last equality comes from recognizing that the integrand is the density of the Dirichlet distribution, for which the beta function $B$ is the normalizing constant.\\
Besides, Gamma function converges extremely rapid to infinity, faster than any exponential functions \cite[Sterling's formula, 6.1.37]{Handbook1972}. Hence, the decay of the above expression shows that $\| \boldsymbol{\varrho}(\beta)\|_F^2 < \infty $ for all $\beta \in \mathbb{R}_+$.
\end{proof}
So naturally, we have the following corollary on the well-posedness of $\mathcal{Z}_a$.
\begin{proposition}
		\label{corollary: welldefined of Z, part2}
		For all measurable bounded function $a$ on $[0,1]\times \mathbb{R}$, the Wiener chaos $\mathcal{Z}_{a}$ is well-defined and has the representation $\mathcal{Z}_{a}= \mathbf{I}( \boldsymbol{\varrho}(a))$.
\end{proposition}
\subsection{Relation between $\mathfrak{Z}$ and $U$-statistics}

 We begin with establishing the relation between partition functions $\mathfrak{Z}$ and $U$-statistics, then we will prove Theorem \ref{theorem: convergence of parition functions}.\\
For convenience, we extend Notation  \ref{notation: (t,x)}  $[t,x]_N$ for a pair $(t,x) \in [0,1]\times \mathbb{R}$ to higher dimensions:
\begin{notation}
	For any pair $(\mathbf{t},\mathbf{x}) \in (0,1]^n \times \mathbb{R}^n$, we let $[ \mathbf{t}, \mathbf{x}]_N$ denote the unique pair $(\mathbf{i},\mathbf{z}) \in [\![1,N]\!]^n \times \mathbb{Z}^n$ such that:
	\begin{itemize}
		\item[i.] $(\mathbf{t},\mathbf{x}) \in \left( \frac{ \mathbf{i}-\mathbf{1}}{N},\frac{\mathbf{i}}{N}\right]\times \left( \frac{\mathbf{z}-\mathbf{1}}{\sqrt{N}},\frac{\mathbf{z}+\mathbf{1}}{\sqrt{N}}\right] $,
		\item[ii.] $\mathbf{i}$ and $\mathbf{z}$ have the same parity.
	\end{itemize}
\end{notation}
\begin{definition}For $n,N \ge 1$, define $p^N_n : [0,1]^n \times \mathbb{R}^n \rightarrow \mathbb{R}$ by
	$$p^N_n( \mathbf{t},\mathbf{x})=2^{-n}p_n (  [\mathbf{t},\mathbf{x}]_N)\mathbb{1}_{\lceil N \mathbf{t} \rceil \in D^N_n},$$
	where $\lceil N \mathbf{t} \rceil $ is the usual ceiling function, that is, for all $\mathbf{x} \in \mathbb{R}^n$ and $\mathbf{z} \in \mathbb{Z}^n$,  $\lceil \mathbf{x} \rceil =\mathbf{z}$ if and only if for all $i$, $\mathbf{z}_i$ is the smallest integer bigger than or equal to $\mathbf{x}_i$ . \label{definition: p^N_n}
\end{definition}
	We observe that the condition $\lceil N \mathbf{t} \rceil \in D^N_n$ implies that $p^N_n$ is identically zero if $n>N$. Besides, we also see that $p^N_n$ is constant on each rectangle in $\mathcal{R}^N_n$, so the average $\overline{p^N_n}=p^N_n$ and in particular, for $\mathbf{i} \in E^N_n, \mathbf{z} \in \mathbb{Z}^n$ such that $\mathbf{i} \leftrightarrow \mathbf{z}$, we have:
	$$p^N_n\left( \frac{\mathbf{i}}{N} ,\frac{\mathbf{z}}{\sqrt{N}}\right)\mathbf{1}_{ \mathbf{i}  \in D^N_n}= 2^{-n}p_{n}( \mathbf{i}, \mathbf{z})\mathbf{1}_{\mathbf{i}\in D^N_n} .$$
	Thus, by definition of $\mathcal{S}^N_n$(see Definition \ref{definition: weighted U-statistics}) ,
	$$\mathcal{S}^N_n( p^N_n)= 2^{-n/2}\sum_{ \mathbf{i} \in D^N_n} \sum_{ \mathbf{z} \in \mathbb{Z}^n} p_n( \mathbf{i},\mathbf{z}) \omega(\mathbf{i},\mathbf{z})A_N( \mathbf{i},\mathbf{z}).$$
	Note that the condition $\mathbf{i} \leftrightarrow \mathbf{z}$ is already handled by $p_n$. 
This leads to the following relation:
\begin{proposition}
	\label{proposition: relation between partition functions and U-statistics}
	For all number real $\beta \in \mathbb{R}$ and positive integer $N \in \mathbb{N}$, the partition functions $\mathfrak{Z}_N$ can rewritten as:
		$$\mathfrak{Z}_N( \beta A_N )= \sum_{n=0}^N 2^{n/2}\beta^n\mathcal{S}^N_n( p^N_n).$$
\end{proposition}
\begin{remark}
	So in particular, $$\mathfrak{Z}_N( N^{-1/4}A_N  )= \sum_{n=0}^N  2^{n/2} N^{-3n/4}\mathcal{S}^N_n( N^{n/2}p^N_n).$$ This equality is useful for our proof of Theorem \ref{theorem: convergence of parition functions}.
\end{remark}
\begin{proof}[Proof of Proposition \ref{proposition: relation between partition functions and U-statistics}]
By definition,
\begin{align}
	&\mathfrak{Z}_N(\beta A_N)
	= \mathbf{E} \left[ \prod_{n=1}^N \big(1+\beta A_N(n,S_n)\omega(n,S_n)\big) \bigg| \omega \right] \nonumber
	\\
	&= \mathbf{E}\left[1+ \sum_{n=1}^N \sum_{\mathbf{i} \in D^N_n} \beta^n\bigg(\prod_{j=1}^n A_N(\mathbf{i}_j,S_{\mathbf{i}_j}) \bigg) \bigg(\prod_{j=1}^n \omega (\mathbf{i}_j,S_{\mathbf{i}_j}) \bigg)  \bigg| \omega \right] \nonumber
	\\
	&= \mathbf{E}\left[1+ \sum_{n=1}^N \sum_{\mathbf{i} \in D^N_n} \sum_{\mathbf{z} \in \mathbb{Z}^n}\beta^n\bigg( \prod_{j=1}^n \mathbb{1}_{ S_{\mathbf{i}}=\mathbf{z}_j}\bigg) \bigg(\prod_{j=1}^n A_N(\mathbf{i}_j,\mathbf{z}_j)\bigg) \bigg(\prod_{j=1}^n \omega (\mathbf{i}_j,\mathbf{z}_j) \bigg) \bigg| \omega \right] \nonumber
	\\
	&=1+\sum_{n=1}^N \sum_{\mathbf{i} \in D^N_n} \sum_{\mathbf{z} \in \mathbb{Z}^n} \beta^n\mathbf{E}\bigg[ \prod_{j=1}^n \mathbb{1}_{ S_{\mathbf{i}}=\mathbf{z}_j}\bigg] A_N(\mathbf{i},\mathbf{z}) \omega( \mathbf{i},\mathbf{z}) \nonumber
	\\
	&= 1+\sum_{n=1}^N \sum_{\mathbf{i} \in D^N_n} \sum_{\mathbf{z} \in \mathbb{Z}^n} \beta^np_n(\mathbf{i},\mathbf{z})A_N( \mathbf{i},\mathbf{z}) \omega(\mathbf{i},\mathbf{z})=1+\sum_{n=1}^N \beta^n2^{n/2} \mathcal{S}^N_n(p^N_n). \label{equation: expansion of Z_N}
\end{align}
Thus,
$\mathfrak{Z}_N(\beta A_N)=1+\sum_{n=1}^N 2^{n/2} \beta^n\mathcal{S}^N_n(p^N_n).$
\end{proof}

Now, we are ready to give  a proof of Theorem \ref{theorem: convergence of parition functions}.
\begin{proof}[Proof of Theorem \ref{theorem: convergence of parition functions}]
	First observe that Theorem \ref{theorem: discrete Chaos to Wiener chaos} and Proposition \ref{propostion: definition of varrho(a) } imply that :
	$$\sum_{n=0}^{\infty}   N^{-3n/4}S^N_n( 2^{n/2} \varrho_n) \xrightarrow[]{(d)} \sum_{n=0}^{\infty}  \tilde{I}_n( \varrho_n 2^{n/2}) =\mathbf{I}( \boldsymbol{\varrho}(\sqrt{2}a))= \mathcal{Z}_{\sqrt{2}a}. $$
	as $N$ converges to infinity.
	Now we show that the difference between this term and $\mathfrak{Z}_N( N^{-1/4}A_N))$ goes to zero as $N$ converges to infinity. Observe that:
	\begin{align*}
		&\sum_{n=0}^{\infty}  N^{-3n/4}S^N_n( 2^{n/2} \varrho_n) - \mathfrak{Z}_N( N^{-1/4}A_N\omega)\\
		&\qquad =\sum_{n=0}^{N}   2^{n/2}N^{-3n/4}S^N_n(  \varrho_n- N^{n/2}p^N_n) + \sum_{n = N+1}^{\infty} N^{-3n/4}S^N_n( 2^{n/2} \varrho_n). 
	\end{align*}
	By Proposition \ref{theorem: properties of U-statistics}, the second term is bounded in $L^2$ by the square root of
	$$\sum_{n=N+1}^{\infty} 2^n c^{2n}\| \varrho_n\|^2_{2}.$$
	which goes to zero as $N \rightarrow \infty$ by Proposition \ref{propostion: definition of varrho(a) }.
	\newline
	For the first term, using again Proposition \ref{theorem: properties of U-statistics}, we note that its $L^2$-norm is bounded above by the square root of
	$$\sum_{n=0}^{N} 2^n c^{2n}\| \varrho_n- N^{n/2}p^N_n\|^2_{2}.$$
	From Lemma \ref{lemma: convergence and boundedness of p^N_n} below, we know that there is a constant $C>0$ such that for all $n \in \mathbb{N}$,
	$$\sup_N \| N^{n/2}p^N_n \|_{2} \le C^n \| \varrho_n\|_{2} \quad \text{ and } \quad \lim_{N \rightarrow \infty} \| \varrho_k- N^{n/2}p^N_n\|^2_{2} = 0.$$
	 Since by Proposition \ref{propostion: definition of varrho(a) },  the sequence $2^{n+1}c^n(1+C^n )\| \varrho_n\|^2_{2}$ is summable, by the dominated convergence theorem, we can easily deduce that:
	\begin{align*}
\lim_{N \rightarrow+\infty} \sum_{n=0}^{N} 2^n c^{2n}\| \varrho_n- N^{n/2}p^N_n\|^2_{2}
 =\sum_{n=0}^{\infty} \lim_{N \rightarrow+\infty}  2^n c^{2n}\| \varrho_n - N^{n/2}p^N_n\|^2_{2} =0.
	\end{align*}
Theorem \ref{theorem: convergence of parition functions} is therefore proved.
\end{proof}

\begin{lemma}
	\label{lemma: convergence and boundedness of p^N_n}
	For all $n$, we have the $L^2$-convergence:
	$$\lim_{N \rightarrow +\infty} \| \varrho_n- N^{n/2}p^N_n\|_{2} =0,$$
	and moreover, there exists a constant $C$ such that for all $n \in \mathbb{N}$,
	$$\sup_N \| N^{n/2}p^N_n \|_{2} \le C^n \| \varrho_n\|_{2}.$$
\end{lemma}
\begin{proof} 	From the local central limit theorem, we deduce that for any fixed $n \in \mathbb{N}$, $N^{n/2}p^{N}_n$ converges almost surely to $\varrho_n$ as $N$ goes to infinity. So by the general Lebesgue dominated convergence theorem \cite[Theorem 19]{Royden2010}, to prove our $L^2$ convergence, it suffices to find a function $g \in L^2([0,1]^n \times \mathbb{R}^n)$ and a  sequence $(g_N, N \in \mathbb{N})$ of functions in $L^2([0,1]^n \times \mathbb{R}^n)$ such that:
	\begin{itemize}
		\item[i.] $\left( N^{n/2}p^{N}_n\right)^2 \le g_N$ for all $N$.
		\item[ii.] $g_N$  converges pointwise to $g$ when $N$ converges to infinity.
		\item[iii.] $\lim_{N \rightarrow \infty} \int_{[0,1]^n\times \mathbb{R}^n} g_N  =  \int_{[0,1]^n\times \mathbb{R}^n} g < \infty$.
	\end{itemize}
	By Definition \eqref{equation: transitions} of $p$ and Sterling's formula (see \cite[Sterling's formula, 6.1.37]{Handbook1972}), we observe that there exists a constant $C$ such that $\sqrt{i}p(i,x) \le C$ for all $i$ and $x$ , therefore:
	$$\sup_{\mathbf{z} \in \mathbb{Z}^n} p_n ( \mathbf{i},\mathbf{z}) \le C^n \prod_{j=1}^n \frac{1}{\sqrt{\mathbf{i}_{j}-\mathbf{i}_{j-1}}}.$$
	From this and by Definition \ref{definition: p^N_n}of $p^N_n$, we have:
	$$\left( N^{n/2}p^{N}_n( \mathbf{t},\mathbf{x}) \right)^2 \le (C/2)^n h\left(  \frac{\lceil N\mathbf{t} \rceil}{N} \right) N^{n/2} p^{N}_n(\mathbf{t},\mathbf{x} ).$$
	where $h(\mathbf{t}) = \prod_{j=1}^n \frac{1}{\sqrt{\mathbf{t}_{j}-\mathbf{t}_{j-1}}} \mathbb{1}_{ \{\mathbf{t} \in \Delta_n\}} $.\\
	Let us choose for all $N$ the function $$g_N(\mathbf{t}, \mathbf{x}):= (C/2)^n h\left(  \frac{\lceil N\mathbf{t} \rceil}{N} \right) N^{n/2} p^{N}_n(\mathbf{t},\mathbf{x} ),$$ 
	and let$$g(\mathbf{t}, \mathbf{x}):= (C/2)^n h( \mathbf{t}) \varrho_n(\mathbf{t}, \mathbf{x}).$$ Clearly, the conditions i. and ii. for the generalized dominated convergence Theorem are sastified. For the last condition, we first notice that:
	$$\int_{[0,1]^n \times \mathbb{R}^n} g( \mathbf{t},\mathbf{x})d\mathbf{t}d\mathbf{x} = (C/2)^n \int_{[0,1]^n} h( \mathbf{t}) d\mathbf{t}.$$
	Then by definition of $p^N_n$, we have the following equalities:
	\begin{align*}
		&\int_{[0,1]^n \times \mathbb{R}^n} g_N( \mathbf{t},\mathbf{x})d\mathbf{t}d\mathbf{x}
		\\
		&=\sum_{ \substack{\mathbf{i} \in [\![ 1,N]\!]^n, \mathbf{z} \in \mathbb{Z}^n: \\ \mathbf{i} \text{ and } \mathbf{z} \text{ have the same parity}}}  \int_{ \left( \frac{ \mathbf{i}-\mathbf{1}}{N},\frac{\mathbf{i}}{N}\right]\times \left( \frac{\mathbf{z}-\mathbf{1}}{\sqrt{N}},\frac{\mathbf{z}+\mathbf{1}}{\sqrt{N}}\right]  } g_N\left(  \mathbf{t}, \mathbf{x}\right) d\mathbf{t} d\mathbf{x}
		\\
		&=\sum_{ \substack{\mathbf{i} \in [\![ 1,N]\!]^n, \mathbf{z} \in \mathbb{Z}^n: \\ \mathbf{i} \text{ and } \mathbf{z} \text{ have the same parity}}} \left( N^{-3n/2}2^n \right) \left[ (C/2)^n h\left( \frac{\mathbf{i}}{N}\right)2^{-n}N^{n/2}p_n(\mathbf{i},\mathbf{z}) \mathbb{1}_{ \mathbf{i} \in D^N_n} \right]
		\\
		&=(C/2)^nN^{-n} \sum_{ \mathbf{i} \in D^N_n} \sum_{\mathbf{z} \in \mathbf{Z}^n} h\left( \frac{\mathbf{i}}{N} \right)p_n( \mathbf{i},\mathbf{z}) \mathbb{1}_{\{\mathbf{i} \text{ and } \mathbf{z} \text{ have the same parity} \}}\\
		&=(C/2)^nN^{-n} \sum_{ \mathbf{i} \in D^N_n}  h\left( \frac{\mathbf{i}}{N} \right) = (C/2)^n  \int_{[0,1]^n}  h\left(  \frac{\lceil N\mathbf{t} \rceil}{N} \right) d\mathbf{t}.
	\end{align*} 

	So, what is left to do is prove that $$\lim_{N\rightarrow \infty}  \int_{[0,1]^n}  h\left(  \frac{\lceil N\mathbf{t} \rceil}{N} \right) d\mathbf{t} =  \int_{[0,1]^n}  h\left(  \mathbf{t} \right) d\mathbf{t} \quad \text{and} \quad \int_{[0,1]^n}  h\left(  \mathbf{t} \right) d\mathbf{t} < \infty.$$
	which is true because $ h\left(  \frac{\lceil N\mathbf{t} \rceil}{N} \right)$ converges pointwise to $h(\mathbf{t})$ for all $\mathbf{t}$ and they form a uniformly integrable sequence of functions in $L^2([0,1]\times \mathbb{R})$. Indeed, the uniform integrability is due to the fact that:
	\begin{align}
		&\int_{\Delta_n} \left[h \left(\frac{\lceil N\mathbf{t} \rceil}{N} \right) \right]^{3/2}d\mathbf{t}= N^{-n}\sum_{ \mathbf{i} \in D^N_n}  \prod_{j=1}^n \left( \frac{\mathbf{i}_{j}}{N}- \frac{\mathbf{i}_{j-1}}{N} \right)^{-3/4} \nonumber
		\\
		&=
		 \sum_{ \mathbf{i} \in D^N_n}   2^{3n/4}\prod_{j=1}^n \left( \frac{2\mathbf{i}_{j}-2\mathbf{i}_{j-1}}{N}\right)^{-3/4} \nonumber
		\\
		& \le \sum_{ \mathbf{i} \in D^N_n}  \int_{ [0,1)^n} 2^{3n/4}\prod_{j=1}^n \left( \frac{\mathbf{i}_{j}-\mathbf{s}_j}{N}- \frac{\mathbf{i}_{j-1} -\mathbf{s}_{j-1}}{N} \right)^{-3/4}d\mathbf{s}  \label{equation: star inequality}
		\\
		&\le  2^{3n/4} \int_{\Delta_n} \prod_{j=1}^n ( \mathbf{t}_j - \mathbf{t}_{j-1})^{-3/4}d\mathbf{t} <\infty. \nonumber
	\end{align}
	Note that in the inequality $\eqref{equation: star inequality}$, we have used the fact that for all positive integer $m\ge 1$ and real numbers $a,b \in [0,1)$: $2m \ge m+ a-b>0$.\\
	For the inequality in the latter part of our lemma, by what we have proved so far, we observe that:
	\begin{align*}
			&\| N^{n/2}p^N_n \|^2_{2} \le (C/2)^n \int_{ \Delta_n} h \left( \frac{\lceil N \mathbf{t} \rceil }{N}\right) d\mathbf{t} \\
			&\le  (C/2)^n 2^{n/2} \int_{\Delta_n} \prod_{j=1}^n ( \mathbf{t}_j-\mathbf{t}_{j-1})^{-1/2}d\mathbf{t}
			\\
			&= C^n2^{-n/2} \int_{\Delta_n} \int_{\mathbb{R}^n} (4\pi)^{n/2} [\varrho_n( \mathbf{t}, \mathbf{x})]^2d\mathbf{t} d\mathbf{x}= C^n(2\pi)^{n/2} \| \varrho_n\|^2_{2},
	\end{align*}
where the second inequality is obtained similarly as \eqref{equation: star inequality}.\\
Hence, we have our desired conclusion.

\end{proof}

\section{Asymptotics of collision measures}\label{section: Limit theorems for collision events}

\subsection{Convergence of  exponential moments}
We first prove Theorem \ref{theorem:uniform integrability of Z_k} on the uniform boundedness of moments of the partition functions. Then we will study the convergence of the exponential moments of $(\frac{1}{\sqrt{N}}\Pi_N;N \in \mathbb{N})$. 
\begin{proof}[Proof of Theorem \ref{theorem:uniform integrability of Z_k}]
	
Let $c$ be a positive number such that $ c \ge \sup_N \| A_N \|_{\infty}$.
Without loss of generality, assume $N$ is sufficiently large (i.e. $N >c^4$) such that the partition function $\mathfrak{Z}_N\left( \frac{1}{N^{1/4}} A_N  \right)$ is a positive random variable.\\
Recall that in page \pageref{equation: heuristic calculation}, we have shown that:
	\begin{align*}
	&\mathbf{E}\left[ \mathfrak{Z}_N\left( \frac{1}{N^{1/4}} A_N  \right)^k \right] =
	\\
	 &	  \mathbf{E}\left[   \prod_{n=1}^N\mathbf{E} \left[  \prod_{i=1}^k \left(1+ \frac{1}{N^{1/4}} A_N (n,S^{(i)}_n)\omega(n,S^{(i)}_n)\right) \bigg| S^{(1)}, S^{(2)},\dots, S^{(k)}\right] \right].
	\end{align*}
\\
Now, define for $n \ge 1$:
\begin{equation}
	\label{equation: definition X_{N,n}}
	X_{N,n}:= \mathbf{E}\bigg[  \prod_{i=1}^k \bigg(1+\frac{1}{N^{1/4}}A_N(n,S^{(i)}_n)\omega(n,S^{(i)}_n) \bigg) \bigg| S^{(1)}, S^{(2)},\dots, S^{(k)} \bigg] -1,
\end{equation}
and \begin{equation}
	\label{equation: definiton T_N}
	T_N := \sum_{n = 1}^N X_{N,n}.
\end{equation}
Because  $\omega$ is a collection of independent Rademacher random variables, we easily notice that $X_{N,n} \ge 0 \quad \mathbf{P}-a.s$, since  $X_{N,n}$ can be represented by:
\begin{equation}
	\label{equation: expansion of X_{N,n}}
	X_{N,n}=\sum_{l=2}^k \sum_{1 \le i_1< i_2<...<i_l \le k}  N^{-l/4}  \prod_{h=1}^lA_N(n,S^{(i_h)}_n)  \underbrace{\mathbf{E} \bigg[ \prod_{h=1}^l \omega(n,S^{(i_h)}_n) \bigg]  }_{\text{is either 0 or 1}} .
\end{equation}
Consequently,
\begin{equation}
	\label{equation: X_{N,N} T_{N}}
	\mathbf{E}\left[ \mathfrak{Z}_N\left( \frac{1}{N^{1/4}} A_N  \right)^k \right] = \mathbf{E}\textsl{}\left[ \prod_{n=1}^N (1+X_{N,n}) \right] \le \mathbf{E}\left[ e^{T_N}\right] .
\end{equation}
where we have used the classical inequality that $\forall x \in \mathbb{R} : 1+x \le e^x $ and $X_{N,n}\ge 0$.

Then, for each $n$, let us introduce  the number $U^{(n)}$ of pairs $(i,j)$ such that $S^{(i)}_n=S^{(j)}_n$, i.e.,
$$U^{(n)}:= \sum_{ 1 \le i <j \le k } \mathbf{1}_{ S^{(i)}_n=S^{(j)}_n}.$$
We observe that on the event $\{U^{(n)}=0\}$, $X_{N,n}$ is equal to zero, and on the event $\{U^{(n)} \ge 1\}$,
\begin{equation}
	\label{equation: U^{(n)}}
	\sum_{1 \le i_1< i_2<...<i_l \le k} \mathbf{E} \bigg[ \omega(n,S^{(i_1)}_n)...\omega(n,S^{(i_l)}_n) \bigg]  \le  \binom{k}{l} \le \binom{k}{l} U^{(n)}.
\end{equation}
Thus, 
\begin{equation}
	\label{equation: T_N,n}
	X_{N,n} \le \sum_{l=2}^k N^{-l/4} c^l\binom{k}{l} U^{(n)} \le (c+1)^k N^{-1/2} U^{(n)}.
\end{equation}
So by combining the inequalities \eqref{equation: X_{N,N} T_{N}} and \eqref{equation: T_N,n}, one sees that:
\begin{align*}
	&\mathbf{E}\left[ \mathfrak{Z}_N\left( \frac{1}{N^{1/4}} A_N \omega \right)^k \right] \le \mathbf{E}\left[e^{T_N}\right]  \\
	\le& \mathbf{E}\bigg[  \exp \bigg(  (c+1)^kN^{-1/2}\sum_{ 1 \le i <j \le k } \sum_{n=1}^N \mathbf{1}_{ S^{(i)}_n=S^{(j)}_n} \bigg)\bigg]
	\\
	=&\mathbf{E}\bigg[  \prod_{ (i,j):1 \le i <j \le k }\exp \bigg(  (c+1)^kN^{-1/2} \sum_{n=1}^N \mathbf{1}_{ S^{(i)}_n=S^{(j)}_n} \bigg)\bigg]
	\\
	\le &\mathbf{E}\bigg[ \exp \bigg( \frac{k(k-1)}{2}(c+1)^kN^{-1/2} \sum_{n=1}^N \mathbf{1}_{ S^{(1)}_n=S^{(2)}_n} \bigg)\bigg],
\end{align*} 
by Hölder's inequality. Besides, using Theorem \ref{theorem: boundedness of exponential} in Appendices, we can prove that for all $\beta \ge 0$:
$$\sup_N \mathbf{E}\bigg[  \exp \bigg(  \beta N^{-1/2} \sum_{n=1}^N \mathbf{1}_{ S^{(1)}_n=S^{(2)}_n} \bigg)\bigg] <+\infty. $$
Thus, we imply the desired conclusion.
\end{proof}
\begin{remark}\label{remark: uniform integrability of exp(T_N)} Using the same argument as in the above proof, one  can see that:
	$$\sup_N \mathbf{E}\left( e^{\beta T_N}\right) <\infty \qquad \forall \beta \ge 0.$$
	Hence, in particular, $\left(e^{\beta T_N},  N \in \mathbb{N}) \right)$ is uniformly integrable.\\
	If we do not care about $U^{(n)}$, we can just have $X_{N,n}\le (c+1)^kN^{-1/2}$.
	This remark will be useful in our proof for Theorem \ref{theorem: convergence of the exponential transform of Pi_N}.
\end{remark}
We now give result on the converence of the exponential moments of $(\frac{1}{\sqrt{N}} \Pi_N , N \in \mathbb{N})$.
\begin{theorem}
	For any bounded positive continous function $f \in \mathcal{C}_{b,+}([0,1]\times \mathbb{R})$, we have:
	$$\mathbf{E} \bigg[ \exp\bigg( \frac{1}{\sqrt{N} }\Pi_N (f)\bigg)\bigg] \xrightarrow[N \rightarrow+\infty]{} \mathbf{E}\bigg[ (\mathcal{Z}_{\sqrt{2f}})^k\bigg].$$
	\label{theorem: convergence of the exponential transform of Pi_N}
\end{theorem} 
\begin{proof}[Proof of Theorem \ref{theorem: convergence of the exponential transform of Pi_N}]

For any  bounded nonnegative continous function $f \in \mathcal{C}_{b,+}([0,1]\times \mathbb{R})$, let 
\begin{itemize}
	\item $A_1,A_2,...$ be a sequence of real functions defined on $\mathbb{N}\times \mathbb{Z}$ such that:
	$$A_N ( n,z):= \sqrt{f}\bigg( \frac{n}{N}, \frac{z}{\sqrt{N}}\bigg) \quad \forall n \in \mathbb{N}, z \in \mathbb{Z}.$$
	\item $a:= \sqrt{f}$ and $c:= \| a\|_{\infty}$.
\end{itemize}Notice that due to the continuity of $f$, $\lim_{N \rightarrow \infty} A([t,x]_N) = a(t,x)$ for all $(t,x) \in [0,1]\times \mathbb{R}$. Thus, $(A_N, N \in \mathbb{N})$ sastifies the condition of  Theorem \ref{theorem: convergence of parition functions} and therefore:
$$\mathfrak{Z}_N( N^{-1/4}A_N ) \xrightarrow[N\rightarrow \infty]{(d)} \mathcal{Z}_{\sqrt{2f}}.$$
Hence, from the uniform integrability in Corollary \ref{corollary: L^k integrability}, we deduce that:
\begin{equation}
	\label{equation: convergence Z_N}
\mathbf{E}\bigg[ \left( \mathfrak{Z}_N( N^{-1/4}A_N ) \right)^k\bigg]  \xrightarrow[N \rightarrow+\infty]{} \mathbf{E}\bigg[(\mathcal{Z}_{\sqrt{2f}})^k \bigg].
\end{equation}
Using again the quantity $X_{N,n}$ defined by \eqref{equation: definition X_{N,n}}, we have shown in \eqref{equation: X_{N,N} T_{N}} that:
$$	\mathbf{E}\bigg[ \left( \mathfrak{Z}_N( N^{-1/4}A_N ) \right)^k\bigg] = \mathbf{E}\left[ \prod_{n=1}^N (1+X_{N,n}) \right]$$
So the convergence \eqref{equation: convergence Z_N} can be rewritten as:
\begin{equation}
	\label{equation: convergence of products}
	\mathbf{E}\left[ \prod_{n=1}^N (1+X_{N,n}) \right]  \xrightarrow[N \rightarrow+\infty]{} \mathbf{E}\bigg[(\mathcal{Z}_{\sqrt{2f}})^k  \bigg]. 
\end{equation}
From Remark \ref{remark: uniform integrability of exp(T_N)}, we know that the sequence $(T_N,N \in \mathbb{N})$ with $T_N=\sum_{n=1}^N X_{N,n}$ satisfies that $\left(e^{\beta T_N},  N \in \mathbb{N}) \right)$ is uniformly integrable and that: $$0\le X_{N,n} \le (c+1)^k N^{-1/2}.$$
Hence, using Theorem \ref{theorem:second relation between products and sums} in Appendix \ref{appendix: sums and products} and the convergence \eqref{equation: convergence of products}, we deduce that:
\begin{equation}\label{equation: covergence of T_n}
	\mathbf{E}\bigg[ e^{T_N}\bigg] \xrightarrow[N \rightarrow+\infty]{} \mathbf{E}\bigg[ (\mathcal{Z}_{\sqrt{2f}})^k\bigg].
\end{equation}
We now investigate the relation between $T_N$ and $\frac{1}{\sqrt{N}} \Pi_{N}(f)$. Observe that:
$$0 \le \frac{1}{\sqrt{N}} \Pi_{N}(f) \le T_N.$$
Indeed, from the expansion \eqref{equation: expansion of X_{N,n}}, we have:
\begin{align*}
	&T_N- \frac{1}{\sqrt{N}} \Pi_{N}(f) =
	\\
	=&  \sum_{n=1}^N \sum_{l=3}^k \sum_{1 \le i_1< i_2<...<i_l \le k}  N^{-l/4} \prod_{h=1}^l A_N(n,S^{(i_h)}_n) \mathbf{E} \bigg[ \prod_{h=1}^l \omega(n,S^{(i_l)}_n) \bigg]  \ge 0.
\end{align*} 
Then following the same arguments that have been used to bound $T_N$ in \eqref{equation: U^{(n)}}  and \eqref{equation: T_N,n}, one can show that:
$$T_N- \frac{1}{\sqrt{N}} \Pi_{N}(f) \le (c+1)^kN^{-3/4} \sum_{1 \le i < j \le k} \sum_{n=1}^N \mathbb{1}_{S_n^{(i)}=S_n^{(j)}} .$$
where the upper bound converges in distribution to $0$ when $N$ converges to infinity.\\
Thus, by applying Lemma \ref{lemma: estimation between two sequences} in Appendices to two sequences $\bigg( e^{ \frac{1}{\sqrt{N}} \Pi_N(f)}  , N \in \mathbb{N}\bigg)$ and $( e^{T_N} , N \in \mathbb{N})$, one can conclude that:
$$\mathbf{E} \bigg[ \exp\bigg( \frac{1}{\sqrt{N} }\Pi_N (f)\bigg)\bigg] \xrightarrow[N \rightarrow+\infty]{} \mathbf{E}\bigg[ (\mathcal{Z}_{\sqrt{2f}})^k\bigg] < \infty.$$
\end{proof}

\subsection{Convergence of collision measures}
We begin by proving the  weak tightness of $\big( \frac{1}{\sqrt{N}} \Pi_N, N \in \mathbb{N} \big)$, then giving the proof of Theorem  \ref{theorem: main theorem}. We refer to \cite[p.118,119]{Kallenberg2017} for the weak tightness. The weak tightness is crucial as it allows us to take convergent subsequences of $(\frac{1}{\sqrt{N}}\Pi_N;N \in \mathbb{N})$(cf. Theorem \ref{theorem: Prokhorov's theorem}).
\begin{theorem}(Prokhorov's theorem, \cite[Theorems 5.1 and 5.2, p.59-60]{Billingsley1999})
	\label{theorem: Prokhorov's theorem}
	Let $E$ be a Polish space and $\Theta$ be a family of probability measures on $E$, then $\Theta$ is tight if and only if $\Theta$ is a relatively compact subset of $\mathcal{P}$, where $\mathcal{P}$ is the topological space of all probability measures on $E$, equipped with the weak convergence topology (see \cite[Chapter 4]{Kallenberg2017} for more details).
\end{theorem}
\begin{remark}
	In our framework of random measures, $E$ is taken to be the space of all positive finite measures on $[0,1]\times \mathbb{R}$, equipped with weak convergence topology.
\end{remark}
\begin{theorem}\label{theorem: tightness}
	The sequence of random measures $\big( \frac{1}{\sqrt{N}} \Pi_N, N \in \mathbb{N} \big)$  is weakly tight.
\end{theorem}
\begin{proof}[Proof of Theorem \ref{theorem: tightness}]
	Let $\mathcal{M}$ denote the set of all finite positive measures on the Polish space $[0,1]\times \mathbb{R}$, and
	\begin{align*}
		&L_m:=\{ \mu \in \mathcal{M}: \| \mu \| \le m \},\\
		&M_m:=\{ \mu \in \mathcal{M}: \text{supp} \mu \subset [0,1]\times [-m,m] \},\\
		&K_m:= L_m\cap M_m.
	\end{align*}
	
	So $K_m$ is a collection of some measures that are uniformly bounded and contained within the same compact set. Thus, by Lemma 4.4 in \cite{Kallenberg2017}, $K_m$ is a weakly relatively compact subset of $\mathcal{M}$. So, by the definition of tightness, it suffices to prove that $$\lim_{m \rightarrow +\infty}  \sup_N\mathbf{P}( N^{-1/2}\Pi_N \not \in K_m)=0,$$
which is true because $$\lim_{m \rightarrow +\infty} \sup_N \mathbf{P}( N^{-1/2}\Pi_N \not \in M_m) =0 \text{  and  }\lim_{m \rightarrow +\infty} \sup_N\mathbf{P}( N^{-1/2}\Pi_N \not \in L_m)  =0.$$
	Indeed, for  $M_m$, we observe that:
	\begin{align*}\mathbf{P}( &N^{-1/2}\Pi_N \not \in M_m)\\
		 &\le \mathbf{P}\bigg(  \sup_{  \substack{ 1 \le n \le N \\ 1 \le i \le k}} |S^{(i)}_n| > m\sqrt{N} \bigg) \le  k\mathbf{P} \bigg( \frac{\sup_{  \substack{ 1 \le n \le N }} |S_n|}{\sqrt{N}} > m\bigg). 
	\end{align*}
	Since the sequence $\left( \frac{1}{\sqrt{N}}\sup_{  \substack{ 1 \le n \le N }} |S_n|  , N \in \mathbb{N}\right)$ converges in distribution to a real random variable (by Donsker's theorem), this sequence is tight by Prokhorov's theorem \cite[Theorem 5.2, p. 60]{Billingsley1999}. Thus,
	\begin{align*}
		\lim_{m \rightarrow +\infty}& \sup_N \mathbf{P}( N^{-1/2}\Pi_N \not \in M_m)  \le k \lim_{m \rightarrow +\infty} \sup_N \mathbf{P}\bigg( \frac{\sup_{  \substack{ 1 \le n \le N }} |S_n|}{\sqrt{N}} > m\bigg)=0.
	\end{align*}
	For $L_m$, we have:
	\begin{align*}
		&\mathbf{P}( N^{-1/2}\Pi_N \not \in L_m) \le \mathbf{P}\left( \sum_{n=1}^N \sum_{1 \le i <j \le k} \mathbb{1}_{S^{(i)}_n=S^{(j)}_n} > m\sqrt{N} \right)  \\
		& \le  \frac{k(k-1)}{2}\mathbf{P} \bigg( \frac{1}{\sqrt{N}}\sum_{n=1}^N  \mathbb{1}_{S^{(1)}_n=S^{(2)}_n} > \frac{2m}{ k(k-1)} \bigg) 
		\\
		&= \frac{k(k-1)}{2}\mathbf{P} \bigg( \frac{1}{\sqrt{N}}\sum_{n=1}^{2N}  \mathbb{1}_{S_n=0} > \frac{2m}{ k(k-1)} \bigg).
	\end{align*}
	Similarly, because $\left( \frac{1}{\sqrt{N}}\sum_{n=1}^{2N}  \mathbb{1}_{S_n=0} ,N \in \mathbb{N} \right)$ also converges in distribution \cite[Theorem 10.1]{Pal2005}, we have $$\lim_{m \rightarrow +\infty} \sup_N\mathbf{P}( N^{-1/2}\Pi_N \not \in L_m)  =0. $$
	Hence the conclusion.
\end{proof}
Now, by combining all results we have shown so far, we can give the proof of Theorem  \ref{theorem: main theorem}.
\begin{proof}[Proof of Theorem \ref{theorem: main theorem}]
	
By Theorem \ref{theorem: tightness} and Prokhorov's theorem \cite[Theorem 5.1]{Billingsley1999}, there exists a random finite positive measure $\mathcal{N}'$ on $[0,1]\times \mathbb{R}$ such that there is a subsequence of $( \frac{1}{\sqrt{N}} \Pi_N, N \in \mathbb{N})$ that converges in distribution to $\mathcal{N}'$. For convenience, assume that $\mathcal{N}'$ is defined on the existing probability space $(\Omega,\mathcal{A}, \mathbf{P})$.\\
Besides, for any $f \in C_{b,+}([0,1]\times \mathbb{R})$, by the proof of Theorem \ref{theorem: convergence of the exponential transform of Pi_N}, it is known that:
$(  e^{\frac{1}{\sqrt{N}} \Pi_N(f)}, N \in \mathbb{N})$ is uniformly integrable. Thus, $\mathbf{E} \left[ e^{\mathcal{N}'(f)} \right]$ is finite and equal to $\mathbf{E}\left[ \left(\mathcal{Z}_{\sqrt{2f}}\right)^k \right]$.\\
We see that to show $\frac{1}{\sqrt{N}} \Pi_N \xrightarrow[N \rightarrow \infty]{wd} \mathcal{N}'$, it suffices to prove that $\mathcal{N}'$ is uniquely defined in distribution.\\
Indeed, let $\mathcal{N}''$ be another random bounded measure on $[0,1]\times \mathbb{R}$ such that there is a subsequence of $( \frac{1}{\sqrt{N}} \Pi_N, N \in \mathbb{N})$ that converges in distribution to it. Assume $\mathcal{N}''$ is also defined on $(\Omega,\mathcal{A}, \mathbf{P})$.\\
In the following, we will prove that $\mathcal{N}'(h) \stackrel{(d)}{=} \mathcal{N}''(h)$ for all $h \in C_b([0,1]\times \mathbb{R})$, then the uniqueness of $\mathcal{N}'$ follows immediately from Lemma 4.7 in \cite{Kallenberg2017}.\\
Let $f,g$ be two continous nonnegative bounded functions on $[0,1]\times \mathbb{R}$. For any two nonnegative numbers $a$ and $b$, $af+bg$ is also a continous bounded nonnegative function. Hence,
$$\mathbf{E} \left[ e^{\mathcal{N}'(af+bg)} \right]=\mathbf{E}\left[ \left(\mathcal{Z}_{\sqrt{2(af+bg)}}\right)^k \right]=\mathbf{E} \left[ e^{\mathcal{N}''(af+bg)} \right].$$
Either, for all $a,b \ge 0$,
$$\mathbf{E} \left[ e^{a\mathcal{N}'(f)+b\mathcal{N}'(g)} \right]=\mathbf{E} \left[ e^{a\mathcal{N}''(f)+b\mathcal{N}''(g)} \right].$$
Or, for all $a,b \ge 0$,
$$a\mathcal{N}'(f)+b\mathcal{N}'(g) \stackrel{(d)}{=} a\mathcal{N}''+b\mathcal{N}''(g).$$
So by Cramer-Wold theorem \cite[Corollary 4.5]{Kallenberg1997}, we have:
$$( \mathcal{N}'(f),\mathcal{N}'(g)) \stackrel{(d)}{=} ( \mathcal{N}''(f),\mathcal{N}''(g)).$$
Then using Cramer-Wold Theorem again, we deduce that $\mathcal{N}'(f-g)\stackrel{(d)}{=}  \mathcal{N}''(f-g)$ for all $f,g \in \mathcal{C}_{b,+}([0,1]\times \mathbb{R})$. So, $\mathcal{N}'(h)\stackrel{(d)}{=}  \mathcal{N}''(h)$ for all $h \in \mathcal{C}_{b}([0,1]\times \mathbb{R})$ because any bounded continous function $h$ can be written as the difference of two continuous bounded nonnegative functions.\\
Thus, we proved that $\frac{1}{\sqrt{N}} \Pi_N \xrightarrow[N \rightarrow \infty]{wd} \mathcal{N}$, where $\mathcal{N}$ is a positive random measure $[0,1]\times \mathbb{N}$ that is uniquely defined in distribution by the following equation for all $f \in \mathcal{C}_{b,+}([0,1]\times \mathbb{R})$:
$$\mathbf{E}( e^{\mathcal{N}(f)})=\mathbf{E}\left[  \left(\mathcal{Z}_{\sqrt{2f}}\right)^k\right] .$$
Finally, the convergence of $(\frac{1}{\sqrt{N}} \Pi'_N, N \in \mathbb{N})$ follows directly from the convergence  of $(\frac{1}{\sqrt{N}} \Pi_N, N \in \mathbb{N})$ and Lemma \ref{lemma: estimation between two sequences} by noticing that $\Pi_N(f) \ge \Pi'_N(f) \ge 0 $ for all $f \in \mathcal{C}_{b,+}([0,1]\times \mathbb{R})$, and
\begin{align*}
&\mathbf{E}\left(  \frac{1}{\sqrt{N}} \| \Pi_N-\Pi'_N \|\right)
\\
&\le \frac{1}{\sqrt{N}}\mathbf{E}\left[ \sum_{n=1}^N \sum_{z \in \mathbb{Z}} \binom{k}{2}\sum_{1 \le i_1\le i_2 \le i_3 \le k}\mathbf{1}_{\left\{ \substack{ S^{(i_1)}_n=S^{(i_2)}_n=S^{(i_3)}_n=z}\right\}} \right]\\
&\le \frac{k^5}{\sqrt{N}} \sum_{n=1}^N \mathbf{P}( S^{(1)}_n=S^{(2)}_n=S^{(3)}_n) \le \frac{k^5}{\sqrt{N}} \sum_{n=1}^N\max_{z \in \mathbb{Z}} ( \mathbf{P}( S^{(3)}_n=z)) \mathbf{P}( S^{(1)}_n=S^{(2)}_n)
\\
&=  \frac{k^5}{\sqrt{N}} \sum_{n=1}^N \frac{1}{2^n}\binom{n}{ \lceil n/2\rceil} \frac{1}{2^{2n}} \binom{2n}{n} \le \frac{k^5}{\sqrt{N}} C^2\sum_{n=1}^N \frac{1}{n} \xrightarrow[N \rightarrow \infty ]{} 0,
\end{align*}
for some constant $C$ such that $ \frac{1}{2^n}\binom{n}{ \lceil n/2\rceil} \le C \frac{1}{\sqrt{n}}$ for all $n \in \mathbb{N}.$ Note that such $C$ exists thanks to Sterling's formula. Hence, our theorem is proved.
\end{proof}
\section*{Acknowledgement}
I am indebted to my Master thesis supervisor Quentin Berger for his invaluable help during my master internship and for introducing me to the techniques of $U$-statistics for solving problems in Statistical Mechanics. My sincere gratitude is reserved for Nicolas Fournier for many crucial discussions. Also, I'm fortunate to have Viet-Chi Tran and Hélène Guérin as my current supervisors, without their reviews and their push, I could not have finished this paper. Finally, this project is supported by Mathematics for Public Health (MfPH) program at the Fields Institute for Research in Mathematical Sciences, Canada, and partly funded by the Bézout Labex, funded by ANR, reference ANR-10-LABX-58.

\begin{appendices}
	\section{On the asymptotic relation between products and sums of independent random variables}
	\label{appendix: sums and products}
	We consider a probability space $(\Omega, \mathcal{A}, \mathbf{P})$.
	For any $N$, let $X_{N} = ( X_{N,n} , n \in \mathbb{N})$ be a sequence of nonnegative random variables  such that the sum $S_N=\sum_{n \ge 1 } X_{N,n}$ is almost surely finite.\\
	Suppose that there exists a sequence of numbers $(c_N, N \in \mathbb{N})$ converging to 0 such that for all $N$, $$c_N \ge \| X_N\|_{\infty} =\sup_n |X_{N,n}|. $$
	Let $$P_N := \prod_{n \ge 1}  (1+X_{N,n}).$$ 
	In this Apprendix, we establish two relations between the sum $S_N$ and the product $P_N$  when $N$ converges to infinity. Note that we do not assume $(X_{N,n}; n \in \mathbb{N} , N \in \mathbb{N})$ to be independent nor identically distributed.\\

	\begin{theorem}(First relation)
		\label{theorem:relation between product and sum}
		For any real random variable $Y$, the following two assertions are equivalent:
	 $$1) \quad S_N \xrightarrow[N \rightarrow +\infty]{(d)} Y \qquad  \qquad 2) \quad P_N \xrightarrow[N \rightarrow +\infty]{(d)} e^{Y}.$$
	\end{theorem}
	\begin{remark}
		There is no moment assumption on $Y$.
	\end{remark}
	\begin{theorem} (Second relation)
		\label{theorem:second relation between products and sums}
		Assume that the sequence $ \left( \exp(S_N), N \in \mathbb{N}\right)$ is uniformly integrable. Then for any real constant $C$, the following two assertions are equivalent:
	$$ 1) \quad \mathbf{E}\left[e^{S_N } \right]\xrightarrow[N\rightarrow+\infty]{} C \quad  \qquad  2) \quad \mathbf{E}\left[ P_N \right]\xrightarrow[N\rightarrow+\infty]{} C.$$
	\end{theorem}
	\begin{proof}[Proof of Theorem \ref{theorem:relation between product and sum}]
	Let us first prove that $1) \Rightarrow 2)$. The inequality $x -\frac{x^2}{2}\le \ln(1+x) \le x$ and the assumption imply that:
	$$0 \le S_N- \ln(P_N)\le \frac{1}{2}\sum_{n \ge 1} X_{N,n}^2 \le c_N S_N \xrightarrow[N \rightarrow+\infty]{ (d)} 0.$$
	Hence, by Slutsky's lemma \cite[Lemma 2.8]{vanderVaart1998}, $\ln(P_N) \xrightarrow[N \rightarrow+\infty]{(d)} Y.$\\
	Let us now prove that $2) \Rightarrow 1)$, we see that for all $x>0, 0 \le x - \ln(1+x) \le x \ln(1+x) $. We deduce
	$$0 \le S_N-\ln(P_N)  \le \sum_{n \ge 1} X_{N,n}\ln(1+X_{N,n}) \le c_N \ln(P_N)  \xrightarrow[N \rightarrow+\infty]{ (d)} 0.$$
	Thus, $S_N  \xrightarrow[N \rightarrow+\infty]{ (d)} Y.$
	The equivalence is proved.
\end{proof}
	\begin{proof}[Proof of Theorem \ref{theorem:second relation between products and sums}]
	For the $1) \Rightarrow 2)$ direction:\newline
	The sequence $(\exp(S_n),n \in \mathbb{N})$ being uniformly integrable, thus there is a subsequence $(n_k, k \in \mathbb{N})$ of $\mathbb{N}$ and a random variable $Z \in L^1$ such that:
	$$\exp{S_{n_k}} \xrightarrow[n \rightarrow\infty]{\text{(d)}} Z \qquad \text{ and } \qquad \mathbf{E}\left[\exp(S_{n_k})\right] \xrightarrow[k \rightarrow \infty]{} \mathbf{E}[Z].$$
	We deduce that $\mathbf{E}[Z]=C$ and by Theorem \ref{theorem:relation between product and sum}, we have $P_{n_k} \xrightarrow[n \rightarrow\infty]{\text{(d)}} Z. $\\
	Besides, the uniform integrability of $(\exp(S_n),n \in \mathbb{N})$ implies the uniform integrability of $(P_N, N \in \mathbb{N})$ ( $0 \le P_N \le e^{S_N}$). So,
	$$ \mathbf{E}[P_{n_k}] \xrightarrow[k \rightarrow \infty]{} \mathbf{E}[Z]=C.$$
	Notice that the uniform integrability and the convergence $\mathbf{E}(S_N) \xrightarrow{N \rightarrow \infty} C$ are still valid if we take any subsequence $(S_{m_i}, i \in \mathbb{N})$ of $(S_N, N \in \mathbb{N})$.\\
	Thus, the result so far implies that for every subsequence $(m_i , i \in \mathbb{N})$ of $\mathbb{N}$, there is a subsequence $ (m_{i_k} , k \in \mathbb{N})$ of $(m_i)$ such that:
	$$ \mathbf{E}[P_{m_{i_k}}] \xrightarrow[k \rightarrow \infty]{} C.$$
	The first implication is proved. The reciprocal is similar.
	\end{proof}
\begin{remark} Note that uniform integrability implies tightness.
\end{remark}
\section{Some auxiliary results on random walks}
Let $(S_n,n \in \mathbb{N}_0)$ be a simple symmetric random walks on $\mathbb{Z}$ and :
\begin{itemize}
	\item[i.] $(X_k, k \in \mathbb{N}_0)$ be a sequence of random variables such that $X_0=0$ and $X_k := \inf \{ N > X_{k-1} : S_N=0 \}$  for all positive integer $k$,
	\item[ii.] $\mathcal{F} := ( \mathcal{F}_k , k \in \mathbb{N}_0)$ be the canonical filtration of the process $(X_k, k \in \mathbb{N}_0)$,
	\item[iii.] $T_k := X_{k}-X_{k-1}$ for all positive integer $k$,
	\item[iv.] $\tau_N := \inf\{ k \ge 0: X_k \ge N \}.$
\end{itemize}
Clearly, by definition, for each $N$, $\tau_N$ is a stopping time with respect to the filtration $\mathcal{F}$ and by Markov's property of $S$,  $(T_k, k \in \mathbb{N})$ is a sequence of indepedent identically distributed random variables.\\
Notice that $T_1$ is the first time after $0$ at which the random walk $S$ returns to the position $0$.  Clearly, this stopping time is well-known. One of its properties is that
\begin{lemma}
	 \label{lemma: first exitting time}
There is a positive constant $C$ such that for all $k \in \mathbb{N}$,  $$\mathbf{P}(T_1=2k) = 2^{-2k+1} \frac{1}{k} \binom{2k-2}{k-1}\ge \frac{C}{k^{3/2}}.$$
\end{lemma}
Indeed, this lemma is just a combination of Theorem 9.2 in \cite{Pal2005} and Sterling's formula.\\
Concerning $\tau_N$, by its definition, we have the following equality which will be useful for our later analysis:
$$\tau_N-1=\sup\{ k \ge 0 : X_k \le N-1 \}= \sum_{n=1}^{N-1} \mathbb{1}_{S_n=0}.$$
In the following is the main theorem of this Section.
\begin{theorem} (Boundedness of exponential moments of local times)
	\label{theorem: boundedness of exponential}
	\newline
	Let $S$ be a random simple walk on $\mathbb{Z}$ starting from $0$, then for any constant $\beta \ge 0$, we have:
	$$ \sup_{N} \mathbf{E}\bigg[ \exp \bigg(  \beta N^{-1/2}\sum_{n=1}^N \mathbb{1}_{S_n=0}\bigg) \bigg] <+\infty.$$
\end{theorem}
This is a corollary of Lemma 4.2 in \cite{Julien2009}. Here, we give an alternative proof.
\begin{proof}

The main idea to prove this theorem is to construct many appropriate martingales to estimate the underlying exponential moment. The construction is as follows, for each $N \in \mathbb{N}$, define:
\begin{itemize}
	\item[i.]$X^N_n:= \sum_{i=1}^n \min(T_i,N).$
	\item[ii.] $\gamma_N  := \inf\{ n \ge 1: X^N_n \ge N \}.$
	\item[iii.] $\lambda_N(\beta ):= -\log \mathbf{E}( e^{ -\beta \min( T_1,N) }) > 0  \quad \forall N \in \mathbb{N}, \beta>0.$
	\item[iii.] $M^N_n:= \exp( - \beta X^N_n + n \lambda_N(\beta)).$
\end{itemize}
Then by noticing that the random variables $T_1,T_2,\dots$ are i.i.d, we see that for each $N$, $(M^N_n, n \in \mathbb{N})$ is a martingale with respect to the filtration $\mathcal{F}$. In addition, because $\forall n,N: X^N_n \ge n$,  $\forall N : \tau_N \le N$. Hence by the optional sampling theorem,  $\forall N \in \mathbb{N} ,\beta>0$,
$$\mathbf{E} \left[\exp( - \beta X^N_{\gamma_N} + \gamma_N \lambda_N(\beta)) \right] =1 .$$
Besides, by definition of $\gamma_N$ and $X^N$, we have:
$$X^{N}_{\gamma_N}= X^{N}_{\gamma_N-1}+\min( T_{\gamma_N},N) \le N+N=2N.$$
Thus, $e^{2\beta} \ge \mathbf{E}( e^{\gamma_N \lambda_N(\beta/N)}) .$\\
Hence, Lemma \ref{lemma: a small small lemma} implies that for all $\beta >0$,
$$\sup_N \mathbf{E} \left[ \exp( \frac{1}{2}c(\beta) \gamma_N/\sqrt{N})\right] <\infty, $$
where $c(\beta):= C \int_{0}^{1/2} \frac{1}{t^{3/2}}(1-e^{-2t\beta})dt$ and $C$ is the constant defined in Lemma  \ref{lemma: first exitting time}.\\
By noticing that $\lim_{\beta \rightarrow \infty} c(\beta)= \infty$ and $\forall N :\tau_N = \gamma_N$, we conclude that for all $\beta>0$:
$$\sup_N \mathbf{E} \left[ \exp( \beta \tau_N/\sqrt{N})\right] <\infty,$$
which is essentially our desired conclusion because $\tau_N -1= \sum_{n=1}^{N-1} \mathbb{1}_{S_n=0}$.
\end{proof}

\begin{lemma}
	\label{lemma: a small small lemma}
The sequence of functions $(\lambda_N, N \in \mathbb{N})$ given in the proof of Theorem \ref{theorem: boundedness of exponential}  sasitifies the following inequality:
$$\liminf_{N \rightarrow \infty} \sqrt{N} \lambda_N(\beta/N) \ge c(\beta),$$
with $c(\beta):= C \int_{0}^{1/2} \frac{1}{t^{3/2}}(1-e^{-2t\beta})dt$, where $C$ is the constant defined in the Lemma \ref{lemma: first exitting time}.
\end{lemma}
\begin{proof}
For any $\beta>0$ and $N\ge 2$, we have:
\begin{align*}
	&1-\mathbf{E}\left[ e^{-\beta \min(T_1,N)/N}\right]\\
	&= \sum_{k=1}^{\lfloor N/2 \rfloor} \mathbf{P}( T_1=2k)(1-e^{-2k\beta/N})+\mathbf{P}( T_1 \ge 2\lfloor N/2 \rfloor+2)(1-e^{-\beta})
	\\
	&\ge  \sum_{k=1}^{\lfloor N/2 \rfloor}\frac{C}{k^{3/2}}(1-e^{-2k\beta/N})
\end{align*}
Thus,
$$\liminf_{N \rightarrow \infty} \sqrt{N}\left(	1-\mathbf{E}\left[e^{-\beta \min(T_1,N)/N}\right]\right) \ge C \int_{0}^{1/2} \frac{1}{t^{3/2}}(1-e^{-2t\beta})dt =c(\beta)>0 .$$
From which, we conclude $\liminf_{N \rightarrow \infty} \sqrt{N} \lambda_N(\beta/N) \ge c(\beta).$
\end{proof}

\section{A useful lemma}
\begin{lemma}
	\label{lemma: estimation between two sequences}
	Let $(U_n) , (V_n)$ be two sequences of positive random variables such that $0 \le U_n \le V_n$ for all $n$, and $V_1,V_2,...$ are uniformly integrable. Then if $\frac{V_n}{U_n} \xrightarrow[n \rightarrow+\infty]{(d)}1$ and $\lim_{n\rightarrow \infty} \mathbf{E}(V_n) =C$, then
	$\lim_{n\rightarrow \infty} \mathbf{E}(U_n) =C.$
\end{lemma}
\begin{proof} The uniform integrability of $(V_n)$ implies the uniform integrability of $(U_n)$. The uniform integrability of $(U_n)$ implies that for every subsequence $(n_k)$ of $\mathbb{N}$, there exists a subsequence $(n_{k_l})$ of $(n_k)$ such that $(U_{n_{k_l}} , l \in \mathbb{N})$ converges in distribution to a random variable $Z$. The convergence of $( \frac{V_n}{U_n}, n \in \mathbb{N})$ implies that $(V_{n_{k_l}} , l \in \mathbb{N})$ also converges in distribution to $Z$. Then, the uniform integrability implies that $\lim_l \mathbf{E}(U_{n_{k_l}}) =C = \lim_l \mathbf{E}(V_{n_{k_l}})$. Hence the conclusion.
\end{proof}

\end{appendices}

\bibliographystyle{plain} 
\bibliography{biblio}
\end{document}